\documentclass[a4paper,12pt]{amsart}
\usepackage[utf8]{inputenc}
\usepackage{amsfonts,amsthm,amsmath}
\usepackage{babel}
\usepackage{tikz-cd}

\usepackage[cmtip,all]{xy}
\xyoption{arc}

\usepackage{pgf,tikz}
\usepackage{pgfplots}
\usetikzlibrary{arrows}

\usepackage{graphicx}
\usepackage{subfigure}
\usepackage{xcolor}
\usepackage{caption}
\usepackage{manfnt}

\usepackage{url}
\usepackage{hyperref}
\usepackage{microtype}

\newcommand{\R}{\mathbb{R}}

\renewcommand{\H}{\mathbb{H}}
\newcommand{\D}{\mathbb{D}}

\renewcommand{\L}{\mathcal{L}}
\newcommand{\cH}{\mathcal{H}}

\newcommand{\B}{\mathcal{B}}
\newcommand{\A}{\mathcal{A}}

\newcommand{\M}{\mathcal{M}}

\newcommand{\X}{\mathcal{X}}

\newcommand{\E}{\mathcal{E}}

\renewcommand{\circ}{\raisebox{-1ex}{$\mathaccent23{\phantom{x}}$}}
\newcommand{\bs}{\backslash}

\newcommand{\isomH}{{\rm SO}^+(n,1)}
\newcommand{\SO}[1]{{\rm SO}(#1)}

\newcommand{\Homeo}{{\rm Homeo}}
\newcommand{\Diffeo}{{\rm Diffeo}}
\newcommand{\Isom}{{\rm Isom}}

\newcommand{\norm}[1]{\parallel \!#1 \! \parallel}

\newcommand{\group}[2]{\raisebox{-2ex} 
{\begin{picture}(28,24)(0,0)
\put(4,15){\vector(1,0){20}}
\put(4,10){\vector(1,0){20}}
\put(10,4){${\scriptstyle #1}$}
\put(10,18){${\scriptstyle #2}$}
\end{picture}}}

\usepackage[title]{appendix}

\newtheorem{theorem}{Theorem}
\newtheorem{lemma}{Lemma}
\newtheorem{corollary}{Corollary}

\newtheorem{proposition}{Proposition}
\newtheorem*{conjecture*}{Conjecture}
\newtheorem{question}{Question}
\newtheorem*{question*}{Question}

\newtheorem{itheorem}{Theorem}

\theoremstyle{definition}
\newtheorem{definition}{Definition}

\newtheorem{remark}{Remark}

\title{Mostow rigidity for skew solenoidal manifolds}
\author[F. Alcalde]{Fernando Alcalde Cuesta} 
\address{Santiago de Compostela, Spain.}
\email{fernando.alcaldecuesta@gmail.com}
\author[M. Mart\'{\i}nez]{Matilde Mart\'{\i}nez}
\address{Instituto de Matem\'atica y Estad\'{\i}stica Rafael Laguardia, 
         Facultad de Ingenier\'{\i}a, Universidad de la Rep\'ublica, 
         J.Herrera y Reissig 565, C.P.11300 Montevideo, Uruguay.}
\email{matildem@fing.edu.uy}
\author[A. Verjovsky]{Alberto Verjovsky}
\address{Instituto de Matem\'aticas, Universidad Nacional Aut\'onoma de M\'exico, Apartado
Postal 273, Admon. de correos \#3, C.P. 62251 Cuernavaca, Morelos, M\'exico.}
\email{alberto@matcuer.unam.mx}

\thanks{M.M. and A.V. would like to acknowledge Grupo CSIC 883174 (UdelaR, Uruguay) and Proyecto PAPIIT IN103324 (DGAPA, UNAM, México) for their financial support. M.M. also thanks the University of Santiago de Compostela and the Mathematics Research Group GI-2136 for their hospitality.}

\begin{document}

\begin{abstract}
We prove a Mostow rigidity theorem for foliated bundles over closed hyperbolic manifolds of dimension $n \geq 3$ endowed with a completely invariant measure of full support. These include solenoidal manifolds obtained as inverse limits of directed systems of finite coverings of closed hyperbolic manifolds. This theorem then extends to skew solenoidal manifolds for which the action of the holonomy group is twisted by means of a cocycle.
\end{abstract}

\maketitle

\section{Introduction}
\label{Sintro} 

Mostow rigidity theorem is one of the most beautiful results of the late 20th century at the crossroads between geometry and topology. Assume that two closed Riemannian manifolds with constant negative curvature and dimension $n \geq 3$ have isomorphic fundamental groups. Then both manifolds are isometric. 

Other than the original proof by G. D. Mostow \cite{Mostow}, there are several proofs including that of G. Prasad in finite volume \cite{Prasad}, the homological approach by M. Gromov \cite{Gromov}, and the geometric approach by G. Besson, G. Courtois and S. Gallot \cite{Besson-Courtois-Gallot}. Here we will adhere to the clear and transparent vision of W. P. Thurston in his notes about the geometry and topology of 3-manifolds \cite{ThurstonNewBook}. The surveys of M. Bourdon \cite{Bourdon,BourdonMostow} and A. L\"ucker \cite{Lucker} will be also very useful for adapting Mostow's rigidity theorem to foliated category.  

All these proofs share the same central point, namely the construction of a boundary map. Combining topological, geometric, analytic and ergodic arguments, it can be proven that this boundary map is conformal and, therefore, extends to an isometry between the manifolds. 
\medskip 

Our aim in this paper is to extend Mostow rigidity from $n$-manifolds to \emph{solenoidal $n$-manifolds}, i.e second countable locally compact Hausdorff spaces which are locally modeled on the product of a $n$-dimensional open disk and a Cantor set. For a comprehensive study of solenoidal manifolds, see \cite{VerjovskySullivan}. Inverse limits of directed systems of finite coverings of closed $n$-manifolds are the more natural examples of solenoids, named \emph{McCord solenoids}, when all coverings are regular. Our first theorem establishes Mostow rigidity in the slightly broader context of suspensions: 

\begin{itheorem}[Mostow rigidity for hyperbolic suspensions] \label{thm:intro1}
Let $M_1$ and $M_2$ be two connected closed hyperbolic manifolds of dimension $n\geq 3$ with fundamental groups $\Gamma_1$ and $\Gamma_2$.  Let $M_{\rho_1}$ and $M_{\rho_2}$ be the suspension of two representations $\rho_1 : \Gamma_1 \to G$ and $\rho_2 : \Gamma_2 \to G$ into a homeomorphism group $G$ acting on a compact Hausdorff topological space $F$. Let $g_1$ and $g_2$ be two leafwise Riemannian metrics on $M_{\rho_1}$ and $M_{\rho_2}$.
Assume $\rho_1(\Gamma_1)$ preserves an ergodic probability measure on $F$ with full support. 
If $h : M_{\rho_1} \to M_{\rho_2}$ is a leafwise homotopy equivalence, then there is a leafwise isometry $I : M_{\rho_1} \to M_{\rho_2}$ integrable homotopic to $h$.
\end{itheorem}

If $F$ is a Cantor set and $G$ is the isometry group of $F$, this framework is equivalent to that of the inverse limits described above, and if $F=G$ is a Cantor group acting on itself by left translation, it is equivalent to that of McCord solenoids. The main difference betwwen both situations is related to the fact that any McCord solenoid $\X$ is \emph{topologically homogeneous}, i.e. for any two points $x,y \in \mathcal{X}$, there is a homeomorphism $f : \mathcal{X} \to \mathcal{X}$ such that $f(x) = y$ (see \cite{McCord}). Therefore, the above theorem has a special case the following rigidity theorem:  

\begin{itheorem}[Mostow rigidity for inverse limit of hyperbolic manifolds] \label{thm:intro2}
 Let $\X_1$ and $\X_2$ be two $n$-dimensional hyperbolic solenoidal manifolds obtained as the inverse limit of towers of finite coverings of a two connected closed hyperbolic manifolds $M_1$ and $M_2$ of dimension $n \geq 3$.  
If $h : \X_1 \to \X_2$ is a homotopy equivalence, then there exists a leafwise isometry $I : \X_1 \to \X_2$ integrable homotopic to $h$. 
\end{itheorem}

The proof of both theorems is the core of this work (see Section~\ref{Ssuspension} and Section~\ref{Sproofth1}). As in the classical proof, the construction of the boundary map is an essential part of the proof, but the suspension hypothesis facilitates the construction of the boundaries at infinity, on which this map is defined. 
\medskip 

However, we can consider Lie and Riemannian solenoidal manifold obtained as solenoidal versions of Lie and Riemannian foliated manifolds from \cite{Molinobook}. Now the construction of a boundary at infinity (extending Thurston's universal circle \cite{Thurston2} to any dimension) is no longer evident. In fact this question will be studied in a forthcoming second part of this paper \cite{AMVMostowfoliations}. Here, to circumvent this difficulty, we first assume a homogeneity condition that actually implies the conditions of Theorem~\ref{thm:intro1} and Theorem~\ref{thm:intro2} (see Section~\ref{Shomogeneous}). We must observe that Lie solenoidal manifolds are always topologically homogeneous, thus homeomorphic to McCord solenoid according to a theorem by A. Clark and S. Hurder \cite{ClarkHurder}, but they are not always Lie-homogeneous (see Remark~\ref{rem:homogeneous}). In fact, we will show how to construct non-homogeneous Lie and Riemannian solenoidal manifolds from cocycles which are not untwisted. Due to their similarity to skew-products in Ergodic Theory, we call them 
\emph{skew solenoidal manifolds}. Fortunately, the existence of a global uniformisation will imply that the construction of the boundary at infinity and the boundary map is similar to the case of suspensions. This allows us to adapt the arguments of the proof of Theorem~\ref{thm:intro1} to this more general case: 

\begin{itheorem}[Mostow rigidity for hyperbolic skew solenoidal manifolds]  \label{thm:intro3}
Let $\Gamma_1$ and $\Gamma_2$ be two discrete groups, and $G$ be a Cantor group or the isometry group of a Cantor set.
Let $\X_1$ and $\X_2$ be two minimal skew hyperbolic solenoidal manifolds of dimension $n \geq 3$ obtained from two representations 
$$
\rho_1 : \Gamma_1 \to G \text{ and } \rho_2 : \Gamma_2 \to G
$$ 
and two cocycles
$$
\varphi_1 : \Gamma_1 \times G \to \isomH \text{ and } \varphi_2 : \Gamma_2 \times G \to \isomH.
$$
If $h : \X_1 \to \X_2$ is a homotopy equivalence, then there exists a leafwise isometry  $I  : \X_1 \to \X_2$  integrable homotopic to $h$.
\end{itheorem} 

An extension of Mostow's rigidity theorem for Lie and Riemannian foliations by dense hyperbolic leaves of dimension $\geq 3$ will be proved in the aforementioned second part of this paper \cite{AMVMostowfoliations}. 
As we will explain in a final remark, the question of how general skew solenoidal manifolds are among Lie and Riemannian solenoidal manifolds will be treated in a third part of this paper. 
\medskip

%%NO ES CIERTO PORQUE UNA MEDIDA ARMÓNICA ERGODICA QUE NOS ES COMPLETAMENTE INVARIANTE SE LEVANTA EN UNA MEDIDA P-INVARIANTE SINGULAR CON RESPECTO A LIOUVILLE
%{\blue 
%Finally, although this paper focuses on solenoidal manifolds, the main theorem applies to any foliated manifold obtained by the suspension of a representation whose image respects an ergodic probability measure. As we shall see, the ergodic argument used in the proof does not always require the measure to be invariant. In some cases, it is sufficient for the holonomy to respect the class of a measure with full support. Thus, in Appendix~\ref{Approjectivebundles}, we shall prove that Mostow’s rigidity is also valid for a large class of flat projective bundles over closed hyperbolic manifolds of dimension $n \geq 3$.}

\section{Solenoidal manifolds}
\label{Ssolenoid}

\begin{definition} \label{def:solenoidalmanifold}
A second countable locally compact Hausdorff space $\X$ is a \emph{$n$-dimensional solenoidal manifold} if $\X$ is covered by open sets $U_i$ endowed with homeomorphisms $\varphi_i : U_i \to \D^n \times K$ where $\D^n$ is the unit open disk in $\R^n$ and $K$ is the Cantor set. Each distinguished open set $U_i$ splits into \emph{plaques} $\varphi_i^{-1}(\D^n  \times  \{y\})$ associated to the points of a \emph{local transversal} $T_i = \varphi_i^{-1}(\{0\} \times  K)$. If $U_i \cap U_j \neq \emptyset$, the change of charts
$$\varphi_{i} \circ \varphi_j^{-1} : \varphi_j(U_i \cap U_j) \to \varphi_i(U_i \cap U_j)$$
is given by
\begin{equation}
\label{ec:cambio_coor}
\varphi_{i} \circ \varphi_j^{-1}(x,y)
= (\varphi_{ij}(x,y),\gamma_{ij}(y)).
\end{equation}
Each map $\gamma_{ij}$ is a homeomorphism between open subsets of $T_j$ and $T_i$. The solenoidal manifold $\X$ is \emph{smooth} if the maps $\varphi_{ij}(-,y)$ are smooth diffeomorphisms depending continuously on $y$ in the smooth topology. All solenoidal manifolds considered here will be smooth and oriented. Additionally, we will assume that the cover $\mathcal{U} = \{U_i\}_{i\in I}$ satisfies the following conditions:

\begin{enumerate}
\renewcommand{\theenumi}{(A\arabic{enumi}}

\item $\mathcal{U}$ is locally finite, hence finite if $\X$ is compact,

\item each open set $U_i$ is a relatively compact,

\item each plaque of $U_i$ intersects at most one plaque of $U_j$.
\end{enumerate}

\noindent
The atlas $\A= \{(U_i,\varphi_i)\}_{i\in I}$ defines a \emph{lamination} or \emph{foliation} $\L$ of $\X$:  the plaques of $\A$ glue together to form maximal path connected $n$-manifolds called \emph{leaves}. In our context, they are the path connected components of $\X$. The disjoint union $T = \bigsqcup \, T_i$ is a \emph{complete transversal} that meets all the leaves. Each local transversal $T_i$ is a clopen set of $T$. 
%In fact, refining the foliated atlas of $\X$ if necessary, we can assume the equivalence relation defined by the plaques on $T$ is trivial, so $T$ is the Cantor set decomposed into clopen sets $T_i$. 
\end{definition}

If $\X$ is a $n$-dimensional solenoidal manifold, the \emph{tangent space} $T\X$ of $\X$ is a $2n$-dimensional smooth solenoidal manifold whose leaves are the tangent bundles of the leaves of $\X$. It is defined by extending the foliated charts of $\X$, in the same way as we define the tangent bundle of a manifold. Other objects from differential geometry can be defined, using the differentiable structure of the leaves. In the transverse direction, these objects are only continuous. 

\begin{definition}\label{def:metric}
If $\X$ is a solenoidal manifold, a \emph{Riemannian metric} $g$ on $\X$ is a positive definite metric tensor on $T\X$ that varies smoothly along the leaves of $\X$ and continuously in the smooth topology of $\X$.
The metric $g$ is \emph{hyperbolic} if all sectional curvatures are equal to -1. We will abusively denote $\X^1 = T^1\X$ the unit tangent bundle derived from the choice of $g$. 
\end{definition}

All solenoidal manifolds admit Riemannian metrics, since they can be defined in local charts and glued using partitions of unity. A solenoidal manifold endowed with a hyperbolic metric will be called a \emph{hyperbolic solenoidal manifold}. 
\medskip 

If $\X$ is endowed with a leafwise Riemannian metric, then $\X$ always admits an atlas $\A$ with the following additional properties (see \cite{ALM2011}): 

\begin{enumerate}
\renewcommand{\theenumi}{(A\arabic{enumi}}
\setcounter{enumi}{2}

\item each plaque is geodesically convex (and even better, for any point $x \in \X$, the star of $x$ with respect to the plaques of $\A$ is contained in a geodesically convex set,

\item boundaries of plaques intersect transversely or not at all.
\end{enumerate}

\noindent
Both properties are needed to replace atlases by box decompositions in the sense of the following definition: 

\begin{definition}\label{def:boxdecomposition}
A family $\B = \{(B_i,\varphi_i)\}_{i \in I}$ is a \emph{box decomposition} of $\X$ if $B_i$ is compact subset of $\X$, named a \emph{flow box}, endowed with a homeomorphism
$$
\varphi_i : B_i \longrightarrow \bar \D^n \times T_i,
$$
where $\bar \D^n$ is the unit closed disk in $\R^n$, $T_i$ is a clopen set in $T$, and 
$$
\partial^v \! B_i = \varphi_i^{-1} \big( \partial \bar \D^n \times T_i \big)
$$ 
is the \emph{vertical boundary} of $B_i$, verifying:  

\begin{enumerate}
\renewcommand{\theenumi}{(B\arabic{enumi}}

\item the family $\B$ covers $\X$,

\item the family $\B$ is locally finite, hence finite $\X$ is compact,  

\item if $B_i \cap B_j \neq \emptyset$, then 
$$
B_i \cap B_j = \partial^v \! B_i \cap \partial^v \! B_j
$$
is homeomorphic to the product of some connected compact subset of $\partial \bar \D^n$ and the intersection of the clopen sets $T_i$ and $T_j$, equal to $T_i$ or $T_j$, 

\item the change of chart is given by the equation
$$
\varphi_{i} \circ \varphi_j^{-1}(x,y)
= (\varphi_{ij}(x,y),\gamma_{ij}(y))
$$
similar to \eqref{ec:cambio_coor}, but $\gamma_{ij}$ is now a homeomorphism from $T_i \cap T_j$ to itself.
\end{enumerate}

\noindent
 of plaques meets in a common face. In this case, we may suppose that the maps $\varphi_{ij}$ are linear and we will say that $\B$ is a \emph{simplicial box decomposition} of $X$. The complete transversal $T = \bigsqcup T_i$ is called the \emph{axis} of $\B$.
\end{definition}

%The existence of box decompositions with geodesically convex plaques and simplicial box decompositions for solenoidal manifolds (of class $C^1$) has been proved in Theorem 2.11 and Theorem 1.2 of \cite{ALM2011}.
The existence of simplicial box decompositions for solenoidal manifolds of class $C^1$ was proven in \cite[Theorem 1.2]{ALM2011} based on the prior construction of box decompositions with geodesically convex plaques.

\section{Suspensions and inverse limits}
\label{Ssuspension}

Many important examples of solenoidal manifolds arise as {\em suspensions} or, equivalently, as laminations of fibre bundles that are transverse to the fibres. See Definition~\ref{def:suspension} for the  meaning of transversality in our context.
We will start here by recalling the classical construction of suspensions, see also \cite{Godbillon} and \cite{Hector-HirschA}. Next, in Section~\ref{Shomogeneous}, we will see some conditions that appear to be more general but nevertheless fall within this category. Our rigidity theorem will be first proved in this context (see Section~\ref{Sproof}), and later will be extended to more general solenoidal manifolds called \emph{skew} to distinguish them from suspensions (see Section~\ref{Snonhomogeneous}). 

\subsection*{Suspensions.} Let $M$ be a $n$-dimensional connected closed manifold and $\Gamma=\pi_1(M,x_0)$ its fundamental group with base point $x_0\in M$. If $p:\tilde M\to M$ is the universal covering of $M$, then $\Gamma$ acts freely and properly discontinuously by deck transformations on $\tilde M$, and $p$ is the projection to the quotient.
\medskip 

Let $F$ be a \emph{space}, $G$ be a \emph{group} acting on $F$, and $\rho:\Gamma\to G$ be a \emph{representation}, that is, a group homomorphism from $\Gamma$ to $G$. If $F$ is a topological space, then $G =\Homeo (F)$ is the group of homeomorphisms; if $F$ is a smooth manifold, then $G = \Diffeo(F)$ is the group of smooth diffeomorphisms; if $F$ is a metric space, then $G=\Isom (F)$ is the group of isometries, and finally if $F$ is a topological group, then $G=F$ acting on itself by left translation.  
\medskip 

The group $\Gamma$ acts diagonally on $\tilde M\times F$ as follows: for each $\gamma\in\Gamma$, 
\begin{equation}  \label{eq:diagonal}
 \gamma.(x,y)=(\gamma x,\rho(\gamma)(y)),
\end{equation}
where the action on the first factor is simply the action by deck transformations. This action preserves the horizontal lamination $\cH$ whose leaves are the horizontal manifolds $\tilde M \times \{y\}$ with $y\in F$.

\begin{definition} \label{def:suspension}
 The \emph{suspension of $\rho$} is the quotient
 $$M_\rho=\Gamma\backslash(\tilde M\times F),$$
 and $q:\tilde M\times F\to M_\rho$ is the projection to the quotient.
\end{definition}

Notice that %the horizontal lamination 
$\cH$ induces a $n$-dimensional lamination $\L_\rho$ on $M_\rho$,
%, so  $M_\rho$ is a laminated or foliated space.
which becomes a laminated or foliated space.
Let \mbox{$p_1:\tilde M\times F\to \tilde M$} be the projection onto the first factor $p_1:(x,y)\mapsto x$. This map induces a locally trivial fibration $\pi:M_\rho\to M$ with fibre $F$ such that the following diagram commutes:

\begin{center}
\begin{tikzcd}
&  & \tilde M\times F \arrow[ld,"p_1" '] \arrow[rd,"q"] \arrow[rr,"p_2"] & & F\\
\phantom{M} & \tilde M   \arrow[rd,"p"] & & M_\rho \arrow[ld,"\pi"] & \\
&  & M. & & 
\end{tikzcd}
\end{center}
We will use the local structure of $M_\rho$ to say that the lamination $\L_\rho$ is \emph{transverse} to the fibration $\pi$.
The projection onto the second factor $p_2: (x,y) \mapsto y$ defining the horizontal lamination $\cH$ is the \emph{developing map} of $\L_\rho$.  Furthermore, for each $y \in F$, the restriction of $q$ to $\tilde M \times \{y\}$ is a regular covering of the leaf $L_y = q(\tilde M \times \{y\})$ with deck transformation group $\Gamma_y = \{\gamma \in \Gamma : \rho(\gamma)(y)= y\}$. For convenience, we will refer indistinctly to the leaves of $\L_\rho$ or $M_\rho$. If both $M$ and $F$ are compact, then $M_\rho$ is also compact. If the action of $\Gamma$ on $F$ has a dense orbit, then $M_\rho$ is connected.
\medskip

If $M$ has a Riemannian metric $g$, the pull-back $\pi^*g$ endows
each leaf of $M_\rho$ with a Riemannian metric. We say that $\pi^*g$ is \emph{leafwise Riemannian metric} on $M_\rho$. If $(M,g)$ is a hyperbolic, then all leaves in $M_\rho$ are hyperbolic when considered with the metric $\pi^*g$ and $M_\rho$ becomes a \emph{hyperbolic foliated space}. The question that we will address is:

%\bigbreak
%\begin{center}
% {\em Does $M_\rho$ admit other hyperbolic structures?}
%\end{center}
%\bigbreak

\begin{question*} 
Does $M_\rho$ admit other hyperbolic structures? Namely, are there leafwise Riemannian metrics on $M_\rho$ for which all leaves have curvature -1 and that are not isometric to $\pi^*g$?
%\end{center}
\end{question*}

%Namely, 
%\begin{center}
% {\em Are there metrics on $M_\rho$ for which all leaves have\\ curvature -1 and that are not isometric to $\pi^*g$?}
%\end{center}
%\bigbreak

In dimension $n=2$ the answer is straightforward. Since the base $M$ admits different hyperbolic structures, their pull-backs are different hyperbolic structures on $M_\rho$. Furthermore, there are hyperbolic metrics on $M_\rho$ which are not isometric to a pull-back metric. In other words, the \emph{Teichm\"uller space} of the foliated space $M_\rho$ is strictly larger than that of $M$.  A detailed analysis can be seen in \cite{BurgosVerjovsky2020} and \cite{VerjovskySullivan} for solenoidal surfaces, and also in \cite{AlvarezSmith2022} and \cite{Deroin} for surface laminations. 
\medskip 

The aim of this section is to show that, just as in the case of manifolds, the situation for $n\geq 3$ is very different, proving the following theorem (named Theorem A in the introduction): 

\begin{theorem}%[Mostow rigidity for hyperbolic suspensions]
 \label{thm:Mostowsuspension}
Let $M_1$ and $M_2$ be two connected closed hyperbolic mani\-folds of dimension $n\geq 3$ with fundamental groups $\Gamma_1$ and $\Gamma_2$.  Let $M_{\rho_1}$ and $M_{\rho_2}$ be the suspension of two representations $\rho_1 : \Gamma_1 \to G$ and $\rho_2 : \Gamma_2 \to G$ into a homeomorphism group $G$ acting on a compact topological space $F$. Let $g_1$ and $g_2$ be two leafwise Riemannian metrics on $M_{\rho_1}$ and $M_{\rho_2}$.
Assume $\rho_1(\Gamma_1)$ preserves an ergodic probability measure on $F$ with full support. 
If $h : M_{\rho_1} \to M_{\rho_2}$ is a leafwise homotopy equivalence, then there is a leafwise isometry $I : M_{\rho_1} \to M_{\rho_2}$ integrable homotopic to $h$.
\end{theorem}

Before stating a natural corollary of this theorem, we will clarify the meaning of the condition imposed on the map $h$:

\begin{definition} A \emph{leafwise homotopy equivalence} $h : M_{\rho_1} \to M_{\rho_2}$ is a homotopy equivalence, with homotopy inverse $h'$, such that 
\medskip 

\noindent
(i) $h$ and $h'$ are \emph{foliated maps} preserving the laminations of $M_{\rho_1}$ and $M_{\rho_2}$, 
\medskip 

\noindent
(ii) the maps $h' \circ h$ and $h \circ h'$ are homotopic to the identity by means of \emph{integral homotopies} along the leaves. 
\medskip 

\noindent
Notice that the first condition only states that the leaves of $M_{\rho_1}$ are sent to leaves of $M_{\rho_2}$ by $h$ (and conversely by $h'$), but the existence of an integral homotopy implies that $h' \circ h$ (resp. $h \circ h'$) leaves each leaf of $M_{\rho_1}$ (resp. $M_{\rho_2}$) invariant.  
\end{definition}

\begin{corollary} \label{cor:Teichmullersuspension}
 Let $M$ be a connected closed hyperbolic manifold of dimension $n\geq 3$ with fundamental group $\Gamma$. Let $M_\rho$ be the suspension of a representation $\rho:\Gamma\to G$ into a homeomorphism group $G$ acting on a compact space $F$. If $\rho(\Gamma)$ preserves an ergodic probability measure on $F$ with full support, then $M_\rho$ admits a unique hyperbolic metric, namely, the pull-back to $M_\rho$ of the unique hyperbolic metric on $M$. In other words, the Theichmüller space of $M_\rho$ contains only one point. 
\end{corollary}

\begin{remark}[About the hypothesis of Theorem~\ref{thm:Mostowsuspension}] \label{rem:ergodic}
%Let us first observe that the representation $\rho_1$ is faithful if the common holonomy covering of the leaves is simply connected. 
Here we specify the assumptions in the two fundamental cases:
\medskip 

\noindent
 (1) If $F = G$ acting on itself by left translation,  then 
\medskip 

\noindent
 (i) the representation $\rho_1$ is faithful if and only if the leaves of $M_{\rho_1}$ are simply connected, 
\medskip 

\noindent
(ii) the action of $\rho_1(\Gamma_1)$ on $G$ is ergodic (with respect to the Haar measure) if and only if $\rho_1(\Gamma_1)$ is dense in $G$ (see \cite[Lemma 2.2.13]{Zimmer}).
\medskip 

\noindent
Thus, replacing $G$ with the compact subgroup $\overline{\rho_1(\Gamma_1)}$, we can always assume that 
%the leaves of $M_{\rho_1}$ are dense, their common universal covering is contractible, and the invariant measure has full support. 
the leaves of $M_{\rho_1}$ are dense and simply connected, and the Haar measure is ergodic with full support.
\medskip 

\noindent
(2) If $G = \Isom (F)$, assuming that the leaves of $M_{\rho_1}$ are dense, S. Matsumoto proves in \cite{Matsumotoequi} that there is a unique ergodic probability measure on $F$ which is invariant by the action of $\rho_1(\Gamma_1)$.  In fact, in our context, we can deduce from a theorem by R. Sacksteder \cite[Theorem 3]{Sacksteder} that $\rho_1(\Gamma_1)$ preserves an (ergodic) probability measure with full support if the leaves of $M_{\rho_1}$ are dense. 
\end{remark}

\begin{remark}[About the laminations] \label{rem:homogeneous}
(1) If $F = G$ acting on itself by left translation, the suspension of a representation $\rho_1 : \Gamma_1 \to G$ is an example of $G$-Lie lamination in the sense of Definition~\ref{def:Lie} which adapts the notion of \emph{Lie foliation} (see \cite{Fedida}, \cite{Godbillon} and \cite{Hector-HirschA}) to the solenoidal context. The leaves are dense and simply connected if and only if the representation $\rho_1$ is faithful and its image is dense in $G$. According to the definition given by \'E. Ghys in \cite[Appendix E]{Molinobook}, such a lamination is said to be \emph{homogeneous} if it is the double coset obtained from a group $H$ that projects onto $G$ with kernel $K$ by the action of a cocompact discrete subgroup of $H$ on the left and the action of a compact subgroup of $K$ on the right. If $G$ is compact (in particular, if $G$ is a Cantor group), it is easy to prove that any homogeneous $G$-Lie lamination is a suspension. However, in general, there are homogeneous $G$-Lie foliations which are not suspensions. All details will be explained in a second part \cite{AMVMostowfoliations} of this work.
\medskip 

\noindent
(2) If $G = \Isom (F)$, the suspension of a representation $\rho_1 : \Gamma_1 \to G$ is an example of Riemannian lamination, similar to the Riemannian foliations studied in \cite{Molinobook} (see also Definition~\ref{def:Lie}). As before, the 
leaves are dense if and only the group $\rho_1(\Gamma_1)$ is dense in $G$.  Furthermore, all the leaves have the same universal covering (see \cite{Reinhart}), which is contractible if the leaves have non-positive sectional curvature. Observe that Molino's theory is particularly straightforward in this context and $M_{\rho_1} = \Gamma_{\rho_1} \bs (\tilde M_1 \times F)$ is a foliated fibre bundle associated to the foliated principal $G$-bundle $E_{\rho_1} = \Gamma_{\rho_1} \bs (\tilde M_1 \times G)$. The representation $\rho_1$ is faithful if and only if the leaves of $E_{\rho_1}$ are simply connected. See Section~\ref{Shomogeneous}. 
\medskip 

\noindent
(3) Skew solenoidal manifolds constructed from cocyles will provide examples of  Lie and Riemannian laminations wich are not homogeneous. Our aim will be to extend Theorem~\ref{thm:Mostowsuspension} to this kind of solenoidal manifolds. See Section~\ref{Snonhomogeneous}. 
\end{remark} 

\begin{remark}[About the finiteness condition] The finiteness of the volume of the hyperbolic manifolds $M_i$ and $M$ is essential in Theorem~\ref{thm:Mostowsuspension} and Corollary~\ref{cor:Teichmullersuspension}. If we take an orientable compact surface $\Sigma$ of genus $g \geq 2$ and $M$ is homeomorphic to the product $\Sigma \times \R$ endowed with a hyperbolic metric, then the group $\Gamma$ is quasi-Fuchsian and the suspension $M_\rho$ of the faithful representation $\rho$ of $\Gamma$ into the profinite completion $G = \hat \Gamma$ admits uncountable many leafwise hyperbolic metrics. 
\end{remark}

%{\blue  Aquí hay un problema bellísimo a atacar. La clasificación de 3-variedades hiperbólicas de tipo finito (i.e, con $\pi_1$ finitamente generado como $\Sigma \times \R$ ) es un resultado excelente de varios autores (Brook, Canary, Minsky, Thurston, . . . ), es el famoso
%Ending lamination theorem ($https://en.wikipedia.org/wiki/Ending-lamination-theorem$), pero esto (y las medidas de Patterson-Sullivan) lo tedriamos que discutir por zoom. Creo que tendriamos una versión de este teorema para variedades solenoidales de tipo McCord sobre 3-variedades hiperbólicas de tipo finito. Suena fascinante.}

Theorem~\ref{thm:Mostowsuspension} will be proved in the next section. The rest of this section consists in formulating the analog (named Theorem B in the introduction) for inverse limits of hyperbolic manifolds and describing the foliated structure of their unit tangent bundles. 

\subsection*{Inverse limits.} Given a $n$-dimensional connected closed manifold $M$, let us consider a directed family of finite coverings of $M$, that is, a family of $n$-dimensional manifolds $M_\alpha$ and non-isomorphic coverings $\pi_\alpha:M_\alpha\to M$ such that the $\alpha$ belong to an infinite set $A$ equipped with the order $\alpha\leq \beta$ if and only if there exists a finite covering 
$$p_{\alpha\beta}:M_\beta\to M_\alpha$$
such that $\pi_\beta = \pi_\alpha \circ p_{\alpha\beta}$.
The inverse limit 
\begin{align*}
\X & = \varprojlim \; \{(M_\alpha, \pi_\alpha)\}_{\alpha\in A} \\
& = 
\{ (x_\alpha) \in \prod_{\alpha \in A} M_\alpha : p_{\alpha\beta}(x_\beta) = x_\alpha \text{ whenever } \alpha\leq \beta \}
\end{align*}
is an $n$-dimensional solenoidal manifold. There is a natural projection $\pi : \X \to M$ where $\pi(x) = x_0$ is the component of $x$ in $M$, and the leaves are transverse to the fibres. In other words, the $n$-dimension solenoidal manifold $\X$ is the suspension of a representation $\rho : \Gamma \to G$ where $\Gamma = \pi_1(M,x_0)$ and $G = \Homeo (F)$ is the homeomorphism group of the Cantor set $F$. See \cite{VerjovskySullivan} for a detailed description. 
\medskip 

If all coverings $p_{\alpha\beta}$ are regular, M. C. McCord proved in \cite{McCord} that the inverse limit $\mathcal{X}$ is \emph{topologically homogeneous}, i.e. for any two points $x,y \in \mathcal{X}$, there is a homeomorphism $f : \mathcal{X} \to \mathcal{X}$ such that $f(x) = y$. Equivalently $\mathcal{X}$ is the suspension of a representation $\rho : \Gamma \to G$ where $G$ is a Cantor group acting on itself by left translation, and we say $\X$ is a \emph{McCord solenoid} over $M$. See also \cite{VerjovskySullivan}.
A theorem by Clark and Hurder \cite{ClarkHurder} proves that any topologically homogeneous solenoidal manifold is homeomorphic to a McCord solenoid obtained as inverse limit of finite regular coverings.
\medskip 

\begin{remark} \label{rem:riemannian}
By replacing the family $\{(M_\alpha, \pi_\alpha)\}_{\alpha\in A}$ by a cofinal set, we can define $\X$ as the inverse limit of a 
tower of non-trivial finite coverings 
%\begin{center}
%\begin{tikzcd}
%\cdots  \phantom{x} M_{m+1} \arrow[rr, "\pi_{m,m+1}"] & &  M_m \arrow[rr,"\pi_{m-1,m}"] & & \phantom{x} \cdots \phantom{x} \arrow[r] & M_1 \arrow[r, "\pi_{0,1}"] & M_0 = M.
%\end{tikzcd}
%\end{center}
$$
\cdots M_{m+1} \overset{p_m}{\longrightarrow} M_m \longrightarrow \cdots \longrightarrow M_1 \overset{p_0}{\longrightarrow} M_0 = M.
$$ 
Let $d$ be the distance on the Cantor set $F= \pi^{-1}(x_0)$ given by 
$$
d((x_m),(y_m)) = 2^{- N((x_m),(y_m))}
$$
where $N((x_m),(y_m))$ is the smallest integer such that $x_m \neq y_m$, or the distance is zero if such a integer does not exist and then $(x_m) = (y_m)$.  As the covering maps $\pi_{m+1} = p_m \circ \cdots \circ p_0$ have the path lifting property, this distance is preserved by the action of the holonomy group $\rho(\Gamma)$. Then the representation $\rho$ factorizes through $G = \Isom (F)$. Now, by applying Arzelà-Ascoli theorem, we know that $G$ is compact endowed with the compact-open topology (which coincides with the topology of uniform convergence). 
\end{remark}

The above observations show that the inverse limit of hyperbolic manifolds are a particular case of suspensions where the fibres are homeomorphic to the Cantor set. Now we can reformulate Theorems~\ref{thm:Mostowsuspension} and Corollary~\ref{cor:Teichmullersuspension} in this context, where we apparently suppress some conditions. However we will immediately see how the new statements can be derived from Theorems~\ref{thm:Mostowsuspension} and Corollary~\ref{cor:Teichmullersuspension}.Next, we specify the statement of this result that was already presented as Theorem B in the introduction.

\begin{theorem}%[Mostow rigidity for inverse limit of hyperbolic manifolds] 
\label{thm:Mostowinvlimit}
 Let $\X_1$ and $\X_2$ be two n-dimensional hyperbolic solenoidal manifolds obtained as the inverse limit of towers of finite coverings of a two connected closed hyperbolic manifolds $M_1$ and $M_2$ of dimension $n \geq 3$. 
If $h : \X_1 \to \X_2$ is a homotopy equivalence, then there exists a leafwise isometry $I : \X_1 \to \X_2$  integrable homotopic to $h$. 
\end{theorem}

\begin{remark}[Reducing Theorem~\ref{thm:Mostowinvlimit} to Theorem~\ref{thm:Mostowsuspension}.] \label{rem:reducingthm2tothm1}
Let us first observe that, under the hypotheses of Theorem~\ref{thm:Mostowinvlimit}, a homotopy  equivalence $h : \X_1 \to \X_2$ always preserves the laminations of $\X_1$ and $\X_2$ since the leaves are just the path connected components of $\X_1$ and $\X_2$. 
%The hypothesis about the holonomy covering of the leaves of $\X_1$ imposed in Theorem~\ref{thm:Mostowinvlimit} is the same as the one imposed in Theorem~\ref{thm:Mostowsuspension}. The meaning of this hypothesis will be explained in Remark~\ref{rem:holonomy}. 
To verify the existence of an invariant measure of full support imposed in Theorem~\ref{thm:Mostowsuspension}, let us distinguish the two type of inverse limits: 
\medskip 

\noindent
1) Any McCord solenoid $\X_1$ over $M_1$ with simply connected leaves is the suspension of a faithful representation $\rho_1$ of $\Gamma_1$ into its profinite completion $G_1 = \hat \Gamma_1$. Then the Haar measure on $G_1$ is preserved by the left action of $\rho_1(\Gamma_1)$ and this action is indeed uniquely ergodic. See also Remark~\ref{rem:ergodic}. In fact, as proved in \cite{VerjovskySullivan}, any other McCord solenoid over $M_1$ is the quotient by a closed normal subgroup of $G_1$. Thus the above condition about the homotopy of the leaves is unnecessary. 
\medskip 

\noindent
2) For any inverse limit of a tower of finite coverings, according to a theorem by Sacksteder \cite{Sacksteder}, there exists a probability measure of full support on the Cantor set $F$ which is invariant by the action of $\rho_1(\Gamma_1)$. As observed in \cite{Matsumotoequi}, this action is also uniquely ergodic.
\medskip 

\noindent
Therefore the proof of Theorem~\ref{thm:Mostowinvlimit} reduces to the proof of Theorem~\ref{thm:Mostowsuspension}.
\end{remark}

\begin{corollary} \label{cor:Teichmullerinvlimit}
The Teichmüller space of the inverse limit of any tower of coverings of a connected closed hyperbolic manifold $M$ of dimension $n\geq 3$ contains only one point. 
\end{corollary}

\subsection*{Foliated structure of the unit tangent bundle.} Similarly to any solenoidal manifold, the unit tangent bundle of a lamination (endowed with a leafwise Riemannian metric) admits a foliated structure whose leaves are the unit tangent bundles of the leaves. We will deal with the two previous situations here.
\medskip 

Let $\H^n$ be the hyperbolic space of dimension $n$ and $\isomH$ the group of orientation-preserving isometries of $\H^n$. Let $M$ be a $n$-dimensional hyperbolic manifold obtained as the quotient of $\H^n$ by the action of a torsion-free cocompact discrete subgroup $\Gamma$ of $\isomH$.  Thus $M$ identifies to the double coset space $\Gamma \bs \isomH / \SO{n}$. 

If $M_\rho$ is the suspension of a representation $\rho: \Gamma \to G$, then the diagonal action of $\Gamma$ on the product $\isomH \times F$ is free and properly discontinuous. The quotient 
$$
\M_\rho = \Gamma \bs \big( \isomH \times F \big)
$$
is a principal $\SO{n}$-bundle over 
$$
M_\rho = \Gamma \bs \big(\isomH/\SO{n} \times F \big).
$$ 
The unit tangent bundle 
$$
M^1_\rho = T^1\L_\rho = \Gamma \bs \big( \isomH/\SO{n-1} \times F \big).
$$

Notice that $\Gamma_\rho = \{ (\gamma,\rho(\gamma)) : \gamma \in \Gamma \}$ is a subgroup of $\isomH \times G$ and then the quotient
$$
\E_\rho = \Gamma_\rho \bs \big( \isomH \times G \big)
$$
has a double structure of principal $\SO{n}$-bundle over 
$$
E_\rho = \Gamma_\rho \bs \big( \isomH/\SO{n} \times G \big) 
$$
and principal $G$-bundle over $\Gamma \bs \isomH$. 
Then $\M_\rho$ is the fibre bundle with fibre $F$ associated to $\E_\rho$ (and therefore $M_\rho$ is the fibre bundle with fibre $F$ associated $E_\rho$). Similarly the unit tangent bundle $M^1_\rho$ is the fibre bundle with fibre $F$ associated to $E^1_\rho = \E_\rho/\SO{n-1}$. 
\medskip 

The \emph{geodesic flow} $g_t$ on $M^1_\rho$ is the flow for which all leaves are invariant, and that restricts to the geodesic flow on each leaf. Thus this flow identifies to the right action of a Cartan subgroup $A=\{a_t\}_{t \in \R}$ of $\isomH$ on $M^1_\rho$ (resp. $E^1_\rho$) induced by the natural right action on $\isomH \times F$ acting trivially on the second factor. Namely 
$$g_t (\Gamma(u,y)) = \Gamma(ua_t,y)$$ 
for all $u \in \isomH$ and all $y \in F$. Notice that the subgroup $\SO{n-1}$ is the centralizer of $A$ in $\SO{n}$. 
\medskip 

The same construction applies to any hyperbolic n-dimensional solenoidal manifold $\X$ given as the inverse limit of a tower of finite coverings 
$$\cdots M_{m+1} \overset{p_m}{\longrightarrow} M_m \longrightarrow \cdots \longrightarrow M_1 \overset{p_0}{\longrightarrow} M_0 = M.$$ 
The unit tangent bundle $\X^1 = T^1\X$ is the inverse limit of the tower of coverings
$$\cdots T^1M_{m+1} \overset{p_m^1}{\longrightarrow} T^1M_m \longrightarrow \cdots \longrightarrow T^1M_1 \overset{p_0^1}{\longrightarrow} T^1M_0 = T^1M,$$ 
where $p_m^1$ denotes the unit tangent map induced by $p_m$.
With the notations of the previous paragraphs, if $\X$ is McCord defined by a representation $\rho : \Gamma \to G$, then the unit tangent bundle
$$
\X^1= \Gamma_\rho \bs \big(\isomH/\SO{n-1} \times G \big),
$$ 
where $\Gamma_\rho$ acts diagonally on $\isomH/\SO{n-1} \times G$.  In general $\X$ is the fibre bundle with fibre $F$ associated to a principal $G$-bundle 
$$
E_\rho = \Gamma_\rho \bs \big( \H^n \times G \big),
$$
where $\Gamma_\rho$ acts diagonally from a representation $\rho: \Gamma \to G = \Isom (F)$. As before, the unit tangent bundle
$\X^1$ is the fibre bundle with fibre $F$ associated to the 
a principal G-bundle $\E_\rho /\SO{n-1}$ where 
$$
\E_\rho  = \Gamma_\rho \bs \big(\isomH \times G\big)
$$
is the quotient of the group $\isomH \times G$ by the (diagonal) action of its subgroup $\Gamma_\rho$.  
\medskip 

The \emph{geodesic flow} $g_t$ on $\X^1$ can be derived from the geodesic flows on the hyperbolic manifolds $M_m$, namely 
$$g_t((u_m)) = (g_t(u_m))$$
for any point $(u_m) \in \X^1$ satisfying $p_{m+1}^1(u_{m+1}) = u_m$ in $T^1 M_m$. This flow identifies again to the right action of a Cartan subgroup $A$ of $\isomH$ induced by the action on $\isomH \times F$ which is trivial on the second factor.

\section{Proof of Theorem~\ref{thm:Mostowsuspension}} \label{Sproof}
\label{Sproofth1}

In this section we will prove Theorems~\ref{thm:Mostowsuspension}~and~\ref{thm:Mostowinvlimit} and their respective corollaries. We will follow the same proof schema that W. P. Thurston used in his proof of Mostow's rigidity theorem \cite{ThurstonNewBook}. We will distinguish four different steps: 
\medskip 

\noindent
(1) The first step will consist in lifting any homotopy equivalence to a suitable covering space where the lamination can be developed as a product.
\medskip 

\noindent
(2) The second step will consist in observing that the lifted map is a leafwise quasi-isometry with uniform constants. 
\medskip 

\noindent
(3) In the third step, we will first construct a boundary map between the boundaries at infinity of both covering spaces, and then deduce from step (2) that the boundary map is quasi-conformal. 
\medskip 

\noindent 
(4) The last step closely follows the last step in Thurston's approach. The goal is to prove that the boundary map is conformal by completing the following observations: 
\medskip 

\noindent
(i) The boundary map is leafwise differentiable a.e. in an appropriate way.
\smallskip 

\noindent
(ii) The excentricity function (which measures how far the boundary map is from being conformal) is uniformly bounded a.e. 
\smallskip 

\noindent
(iii) The excentricity function is constant a.e. 
\smallskip 

\noindent
(iv) The leafwise differential of the boundary map is conformal a.e. and then the boundary map itself is conformal. 
\smallskip 

\noindent
Finally, in that case, the boundary map extends to a leafwise isometry between the covering spaces that induces a leafwise isometry between the original spaces being integrable homotopic to the original homotopy equivalence. 

\subsection*{Homotopy equivalences and bundles morphisms.} 

Let $M_1$ and $M_2$ be two connected closed hyperbolic manifolds of dimension \mbox{$n\geq 3$} with fundamental groups $\Gamma_1$ and $\Gamma_2$.  Let $M_{\rho_1}$ and $M_{\rho_2}$ be the suspension of two representations $\rho_1 : \Gamma_1 \to G$ and $\rho_2 : \Gamma_2 \to G$ into a homeomor\-phism group $G$ acting on a compact Hausdorff topological space $F$. In particular, $M_{\rho_1}$ and $M_{\rho_2}$ are fibre bundles over $M_1$ and $M_2$ with fibre $F$. 
%Observe also that if $\rho_1(\Gamma_1)$ preserves an ergodic measure on $F$ with full support, then the action of $\Gamma_1$ on $F$ has a dense orbit. This implies that $\M_{\rho_1}$, like any other space that is homotopy equivalent to $\M_{\rho_1}$, is connected. 
Observe also that, under the hypotheses of Theorem~\ref{thm:Mostowsuspension}, both spaces $M_{\rho_1}$ and $M_{\rho_2}$ are connected. Our aim is to prove the following result:

\begin{proposition} \label{prop:lifting} 
Any leafwise homotopy equivalence $h : M_{\rho_1} \to M_{\rho_2}$ is  integrable homotopic to a leafwise homotopy equivalence $H : M_{\rho_1} \to M_{\rho_2}$ such that 
\medskip 

\noindent
(i) $H$ is a morphism of fibre bundles, 
\medskip 

\noindent
(ii) $H$ is leafwise smooth. 
\end{proposition} 

To prove the proposition, we start by stating and proving a preliminary lemma: 

\begin{lemma} \label{lem:bundle}
Assume that $\rho_1 : \Gamma_1 \to G$ and $\rho_2 : \Gamma_2 \to G$ are conjugated, that is, there are an isomorphism $\varphi : \Gamma_1 \to \Gamma_2$ and a homeomorphim $k \in G$ such that
$$
\rho_2(\varphi(\gamma_1)) \circ k = k \circ \rho_1(\gamma_1)
$$
for all $\gamma_1 \in \Gamma_1$. Then there is a homotopy equivalence $H : M_{\rho_1} \to M_{\rho_2}$ such that 
\medskip 

\noindent
(i) $H$ is a morphism of fibre bundles with base map $\bar h : M_1 \to M_2$,
\medskip 

\noindent
(ii) $H$ is leafwise smooth, 
\medskip 

\noindent
(ii) $\varphi = \bar h_\ast$.
\end{lemma}

\begin{proof} Since $M_1$ and $M_2$ are $K(\pi,1)$ spaces, there is a homotopy equivalence $\bar h : M_1 \to M_2$ (with homotopy inverse $\bar h' : M_2 \to M_1$) inducing the isomorphism $\varphi : \Gamma_1 \to \Gamma_2$. Then $\bar h$ and $\bar h'$ lift to maps $\tilde h : \tilde M_1 \to \tilde M_2$ and  $\tilde h': \tilde M_2 \to \tilde M_1$ such that $\tilde h \circ \tilde h'$ and $\tilde h' \circ \tilde h$ are equivariantly homotopic to the identity. Now consider the map
$$
\tilde H : \tilde M_1 \times F \to \tilde M_2 \times F
$$
given by 
\begin{equation} \label{eq:homotopyequivalence}
\tilde H  (x,y) = \Big(\tilde h (x) , k(y) \Big),
\end{equation}
where $x \in \tilde M_1 $ and $y \in F$.  This map satisfies the following conditions: 
\medskip 

\noindent
(1) $\tilde H $ is a homotopy equivalence (with homotopy inverse $\tilde H'$ defined in the same way), 
\medskip

\noindent
(2) $\tilde H$ is equivariant for the diagonal actions of $\Gamma_1$ and $\Gamma_2$:
\begin{align*}
\bar h_\ast(\gamma_1) . \tilde H  (x,y) & = \bar h_\ast(\gamma_1) . \Big(\tilde h (x) , k(y) \Big) \\
 & = \Big( \bar h_\ast(\gamma_1) \tilde h (x) , \rho_2\big(\bar h_\ast(\gamma_1)\big) \big(k(y)\big) \Big) \\
 & = \Big( \tilde h (\gamma_1. x)  , k\big(\rho_1(\gamma_1) (y) \big) \Big) \\
 & = \tilde H  \Big( \gamma_1. x  ,  \rho_1(\gamma_1) (y) \Big)  \\
 & = \tilde H \Big( \gamma_1 . \big(x, y\big) \Big).
\end{align*}

\noindent
for all $\gamma_1 \in \Gamma_1$, all $x \in \tilde M_1$ and all $y \in F$. 
\medskip

\noindent
(3) $\tilde H $ preserves the horizontal laminations of 
$\tilde M_1 \times F$ and $\tilde M_2 \times F$. 
\medskip

\noindent
Therefore $\tilde H $ induces a leafwise homotopy equivalence 
$H : M_{\rho_1} \to M_{\rho_2}$ that makes the following diagram commutative: 
\begin{equation*}
\begin{tikzcd}[row sep=7ex,column sep ={10ex}]
\tilde M_1 \times F \arrow[swap]{d}{q_1} \arrow{r}{\tilde H} & \tilde M_2 \times F \arrow{d}{q_2} \\
M_{\rho_1} \arrow{r}{H} & M_{\rho_2}.
\end{tikzcd}
\end{equation*} 
This implies that $H$ is a morphism of fibre bundles
\begin{equation*}
\begin{tikzcd}[row sep=7ex,column sep ={10ex}]
M_{\rho_1} \arrow[swap]{d}{\pi_1} \arrow{r}{H} &  M_{\rho_2} \arrow{d}{\pi_2} \\
M_1 \arrow{r}{\bar h} & M_2.
\end{tikzcd}
\end{equation*}
inducing the homotopy equivalence $\bar h : M_1 \to M_2$ on the base spaces. The morphims $H' \circ H$ and $H  \circ H'$ are homotopic to the identity and then the fibre bundles induced by $\bar h' \circ \bar h$ and $\bar h \circ \bar h'$ are isomorphic to $M_{\rho_1}$ and $M_{\rho_2}$. %PUEDE RESULTAR CONFUSO
%Notice that, since the base spaces $M_1$ and $M_2$ are homeomorphic, the total spaces $M_{\rho_1}$ and $M_{\rho_2}$ are also homeomorphic.
Finally, observe that the maps $\bar h$ and $\bar h'$ can be appropiately approximated to ensure that $H$ and $H'$ are leafwise smooth. 
\end{proof}

\begin{proof}[Proof of Proposition~\ref{prop:lifting}]
In general, any leafwise homotopy equivalence 
$$h : M_{\rho_1} \to M_{\rho_2}$$ 
can be lifted to an equivariant leafwise homotopy equivalence 
$$
\tilde H  : \tilde M_1 \times F \to \tilde M_2 \times F.
$$
As $\tilde H$ preserves the horizontal laminations of $\tilde M_1 \times F$ and $\tilde M_2 \times F$, it induces a continuous map on the second factor $F$. But the hyperbolic manifolds $\tilde M_1$ and $\tilde M_2$ are contractible, diffeomorphic to $\R^n$. This implies that the map $k$ induced by $\tilde H$ on the second factor is a homotopy equivalence. In fact, 
%by the homotopy lifting property of $M_{\rho_1}$ and  $M_{\rho_2}$,
by the unique homotopy lifting property for the covering maps $q_1$ and $q_2$, the map $k$ is a homeomorphism. Then the isomorphism $h_\ast$ between the fundamental groups of $M_{\rho_1}$ and $M_{\rho_2}$ induces an isomorphism $\bar h_\ast$ between $\Gamma_1$ and $\Gamma_2$.
%SUCESION EXACTA CORTA DE HOMOTOPIA
We now have all the ingredients to construct, using Lemma~\ref{lem:bundle}, a leafwise bundle morphism $H : M_{\rho_1} \to M_{\rho_2}$ that is integrable homotopic to $h$. 
\end{proof}

\begin{remark}[Classifying groupoids and bundles] \label{rem:groupoids}
In fact, the construction of the homotopy equivalence $H$ in Proposition~\ref{prop:lifting} can be reformulated from a more conceptual argument involving Morita equivalence for groupoids. To do so, we shall first assume that holonomy coverings of the leaves of $M_{\rho_1}$ are simply connected. Then the holonomy groupoid of $M_{\rho_1}$ becomes isomorphic to homotopy groupoid, that is,  
$$
Hol_{\rho_1} = \Gamma_{\rho_1} \bs \big( \tilde M_1 \times \tilde M_1 \times F \big)
\group{\alpha}{\beta} M_{\rho_1}.
$$ 
Now the foliated space $\tilde M_1 \times F$ defines a groupoid equivalence (in the sense of J. Renault \cite{Renault1982}; see also  \cite{MuhlyRenaultWilliams}) between the holonomy groupoid 
$$
Hol_{\rho_1} = \Gamma_{\rho_1} \bs \big( \tilde M_1 \times \tilde M_1 \times F \big)
\group{\alpha}{\beta} M_{\rho_1}
$$ 
and the transformational groupoid 
$$
\Gamma_1 \times F \group{\alpha}{\beta} F.
$$
Indeed, the orbit space of the left action of $Hol(\L_{\rho_1})$ on $\tilde M_1 \times F$ is homeomorphic to $F$, whereas the orbit space of the right action of the transformational groupoid on $\tilde M_1 \times F$ is leafwise diffeomorphic to $M_{\rho_1}$, and both actions commute. As $\tilde M_1$ is contractible, the holonomy groupoid $Hol_{\rho_1}$ is \emph{classifying} (see \cite{Haefligerclassifiant} and \cite{HectorBC}). 
Given a leafwise homotopy equivalence $h : M_{\rho_1} \to M_{\rho_2}$, the holonomy groupoids 
 $Hol_{\rho_1}$ and  $Hol_{\rho_2}$ are isomorphic and hence the transformational groupoids 
 $$
\Gamma_1 \times F \group{\alpha}{\beta} F \quad \text{and} \quad 
\Gamma_2 \times F \group{\alpha}{\beta} F
$$
are also isomorphic. This means that both groups are isomorphic and their actions are conjugated. Now we can apply Lemma~\ref{lem:bundle} to obtain a homotopy equivalence $H : M_{\rho_1} \to M_{\rho_2}$ which is integrable homotopic to $h$. 

In fact, the fibre bundles $M_{\rho_1}$ and $M_{\rho_2}$ are themselves classifying spaces for the transverse structures represented by $Hol_{\rho_1}$ and $Hol_{\rho_2}$ (in the sense of \cite{Haefligerclassifiant}) under two aditional conditions: (i) the representations $\rho_1$ and $\rho_2$ are faithful, (ii) any element of $\rho_1(\Gamma_1)$ and $\rho_2(\Gamma_2)$ is the identity if it is the identity on some non-empty open subset of $F$.

The condition about the holonomy covering of the leaves has been only used to identify the holonomy groupoid with the homotopy groupoid. The argument above literally applies to the homotopy groupoids of $M_{\rho_1}$ and $M_{\rho_2}$ as the universal covering of the leaves of $M_{\rho_1}$ and $M_{\rho_2}$ is contractible.
\end{remark}

\begin{remark}[Classifying solenoids] \label{rem:holonomy}
Although the above arguments apply to inverse limits such as those appearing in the statement of Theorem~\ref{thm:Mostowinvlimit}, their construction involves constraints on the foliated structure that avoid the use of groupoids and simplify homotopical arguments. 
If $\X_1$ and $\X_2$ are inverse limits of finite regular coverings of $M_1$ and $M_2$, they are McCord solenoids obtained as the suspension of two representations $\rho_1 : \Gamma_1 \to G_1$ and $\rho_1 : \Gamma_1 \to G_1$, where $G_1 = \hat \Gamma_1$ and $G_2 = \hat \Gamma_2$ are the profinite completions of the fundamental groups $\Gamma_1$ of $M_1$ and $\Gamma_2$ of $M_2$. In this setting, if $\X_1$ and $\X_2$ have simply connected leaves, the groups $\Gamma_1$ and $\Gamma_2$ are isomorphic and the representations $\rho_1$ and $\rho_2$ are conjugated. Indeed, if the leaves of the $\X_1$ and $\X_2$ are simply connected, they are contractible, diffeomorphic to $\R^n$. Then the foliated structures of $\X_1$ and $\X_2$ are classifying, determined by the natural left actions of $\Gamma_1$ and $\Gamma_2$ on $G_1$ and $G_2$. Thus $\rho_1$ and $\rho_2$ are faithful and \emph{a fortiori} conjugated. However, as observed earlier in Remark~\ref{rem:reducingthm2tothm1}, we do not need the condition on the leaves to see that $\rho_1$ and $\rho_2$ are conjugated if $\X_1$ and $\X_2$ are homotopically equivalent. 

More generally, if $\X_1$ and $\X_2$ are inverse limit of finite coverings, not necessarily regular, we can consider the foliated spaces $E_{\rho_1}$ and $E_{\rho_2}$ obtained from representations of $\Gamma_1$ and $\Gamma_2$ into the isometry groups $G_1$ and $G_2$ of the fibres of $\X_1$ and $\X_2$. Recall that $\X_1$ and $\X_2$ are fibre bundles associated to the principal bundles $E_{\rho_1}$ and $E_{\rho_2}$. Arguing directly on $\X_1$ and $\X_2$ (or passing to $E_{\rho_1}$ and $E_{\rho_2}$ and applying the previous remark), we similarly deduce that the representations $\rho_1$ and $\rho_2$ are conjugated if $\X_1$ and $\X_2$ are homotopically equivalent. 
\end{remark}

\subsection*{Quasi-isometry.} Let $g_1$ and $g_2$ be two leafwise Riemannian metrics on $M_{\rho_1}$ and $M_{\rho_2}$. We shall assume that both are hyperbolic so that the leaves of $M_{\rho_1}$ and $M_{\rho_2}$ have negative constant curvature equal to $-1$, although this property will not be used in this paragraph. Consider the pull-backs $\tilde g_1 = \pi_1^\ast g_1$ and $\tilde g_2 = \pi_2^\ast g_2$ to the foliated products $\tilde M_1 \times F$ and $\tilde M_2 \times F$. Now we shall prove the foliated version of Lemma 5.9.1 from Thurston's notes \cite{ThurstonNewBook}: 

\begin{proposition} \label{prop:quasi-isometry}
Any leafwise homotopy equivalence $h : M_{\rho_1} \to M_{\rho_2}$ satisfying the condition of Proposition~\ref{prop:lifting} is a leafwise quasi-isometry which lifts to an equivariant quasi-isometry 
$$
\tilde H : \tilde M_1 \times F \to \tilde M_2 \times F
$$ 
with uniform constants. 
\end{proposition}

\begin{proof} By compactness, the restrictions of $h$ and $h'$ to the leaves of $M_{\rho_1}$ and $M_{\rho_2}$ are uniformly Lipschitz with the same Lipschitz constant $C_1$. Therefore the liftings 
$$
\tilde H : \tilde M_1 \times F \to \tilde M_2 \times F \quad \text{and} \quad
\tilde H' : \tilde M_2 \times F \to \tilde M_1 \times F  
$$
also restrict to uniformly Lipschitz maps on the horizontal leaves with the same Lipschitz constant $C_1 > 0$. By definition, the composition $h' \circ h$ and $h \circ h'$ are homotopic to the identity by an integral homotopy. Therefore $\tilde H' \circ \tilde H$ and $\tilde H \circ \tilde H'$ are homotopic to the identity by an integral homotopy (keeping invariant the second coordinate up to homeomor\-phism). Then the restriction of the compositions $h' \circ h$ and $h \circ h'$ are uniformly at bounded distance from the identity with the same bound $C_2 \geq 0$. The compositions $\tilde H' \circ \tilde H$ and $\tilde H \circ \tilde H'$ restrict again to maps at bounded distance of the identity in a uniform way and with the same bound $C_2 \geq 0$. In summary, restricted to the leaves, the homotopy equivalences $h$ and $h'$, as well their lifted maps $\tilde H$ and $\tilde H'$, are quasi-isometries in Gromov's sense (according to the definition given in \cite{Ghys-delaHarpe}) with the same constants $C_1 >0$ and $C_2 \geq 0$. 
\end{proof}

\subsection*{Boundary at infinity.}

For each point $y \in F$, let $\partial^y \tilde M_ 1$ be the visual boundary of $\tilde M_1 \times \{y\}$ with respect to the restriction $\tilde g_1^y$ of $\tilde g_1$ to $\tilde M\times \{y\}$. Let us recall the definition. If $\alpha(t)$ and $\beta(t)$ are geodesic rays in $\tilde M_ 1 \times \{y\}$ with respect to the metric $\tilde g_1^y$, parametrized by arc length, we say that they are \emph{asymptotic} (defining the same point $\alpha (+\infty) = \beta (+\infty)$ in $\partial^y \tilde M_ 1$) if there exists $s>0$ such that the distance between $\alpha(t)$ and $\beta(t+s)$ tends to 0 as $t \to +\infty$. 
%DEFINITION DE strongly asymptotic, EQUIVALENTE A asymptotic PARA ESPACIOS HIPERBOLICOS, PERO NO PARA VARIEDADES DE HADARMARD.
Here the distance is taken with respect to Riemannian metric $\tilde g_1^y$. This defines an equivalence relation between geodesic rays in $\tilde M_1 \times \{y\}$, and 
$\partial^y \tilde M_1$ is the set of equivalence classes. 

If $x\in \tilde M_1$, then any geodesic ray parametrized by arc length in $\tilde M_1$ is asymptotic to one and only one ray starting at $x$, defining a bijection between $\partial^y \tilde M_1$ and the unit tangent space $T^1_{(x,y)} (\tilde M_1 \times \{y\})$ at the point $(x,y)$. This gives $\partial^y \tilde M_1$ a differentiable structure, which is independent of the choice of $x$, such that $\partial^y \tilde M_1$ is diffeomorphic to the sphere $S^{n-1}$.

\begin{definition} \label{def:boundary}
We call \emph{boundary at infinity} of the foliated space $\tilde M_1 \times F$ (with respect to the metric $\tilde g_1$) to the union of visual boundaries
\begin{equation*}
 \partial (\tilde M_1 \times F) = \bigcup_{y\in F}\partial^ y \tilde M_1.
 \end{equation*}
This is a topological space which is leafwise diffeomorphic to $S^{n-1}\times F$. In other words, 
the boundary $\partial (\tilde M_1 \times F)$ is the total space of a trivial fibration over $F$. As the action of $\Gamma_1$ preserves $\tilde g_1$, its fibre is the \emph{universal boundary} $S^{n-1}_\infty$ in the sense of Thurston \cite{Thurston2}, and it is endowed with a canonical conformal structure. Later, in Remark~\ref{rem:tubes}, we shall see how to construct a global trivialisation of $\partial (\tilde M_1 \times F)$ from a global uniformisation of $M_{\rho_1}$.
\end{definition}

%As a straightforward consequence of Proposition~\ref{prop:quasi-isometry}, we obtain the following lemma, which is the foliated version of Proposition 5.9.2 from Thurston's notes  \cite{ThurstonNewBook}: 
As a consequence of Proposition~\ref{prop:quasi-isometry}, we obtain a foliated version of Morse lemma (see \cite{Bourdon} and \cite{Lucker} for a detailed formulation) from which the continuity of the boundary map is derived. Here we follow Thurston's more conceptual approach by adapting Proposition 5.9.2 from his notes \cite{ThurstonNewBook} to our foliated context:

\begin{lemma} \label{lemma:geodesics_go_to_quasi_geodesics}
The homotopy equivalence 
$$
\tilde H : \tilde M_1 \times F \to \tilde M_2 \times F
$$
sends each geodesic $\alpha(t)$ in $\tilde M_1 \times \{y\}$ with respect to the metric $\tilde g_1^y$ to a curve that remains at a bounded distance from a geodesic $\beta(t)$ in $\tilde M_2 \times \{k(y)\}$ with respect to the metric $\tilde g_2^{k(y)}$. \hfill \qed
\end{lemma}

%Morse-Mostow Lemma. Given $K > 0$, there exists a constant $D = D(K) > 0$ such that for any $(K,\varepsilon)$-quasi-geodesic $\alpha$, there exists a unique geodesic $\beta$ within a $D$-neighborhood of $\alpha$.

\begin{proposition} \label{prop:boundarymap}
The homotopy equivalence 
$$
\tilde H : \tilde M_1 \times F \to \tilde M_2 \times F
$$
induces an equivariant boundary map 
$$
\partial \tilde H : \partial (\tilde M_1 \times F) \to 
\partial (\tilde M_2 \times F)
$$
defining a homeomorphism between both boundaries at infinity.
\end{proposition} 

\begin{proof} As in the proof of Corollary 5.9.3 from Thurston's notes \cite{ThurstonNewBook}, the definition of $\partial \tilde H$ follows inmediately from the previous Lemma~\ref{lemma:geodesics_go_to_quasi_geodesics}. Indeed, for any geodesic $\alpha(t)$ in $\tilde M_1 \times \{y\}$, the map $\partial \tilde H$ sends the limit point $\big(\alpha(+\infty),y\big)$ to the point $\big(\beta(+\infty),k(y)\big)$ where $\beta$ is the unique geodesic in $\tilde M_2 \times \{k(y)\}$ which is at a bounded distance from the image of $\alpha$ by $\tilde H$. As $\Gamma_1$ and $\Gamma_2$ act isometrically on the spaces $\tilde M_1 \times F$ and $\tilde M_2 \times F$, endowed with the leafwise Riemann metrics $\tilde g_1$ and $\tilde g_2$, the boundary map becomes equivariant relatively to the (leafwise) conformal actions of $\Gamma_1$ and $\Gamma_2$ on theirs boundaries $\partial (\tilde M_1 \times F)$ and $\partial (\tilde M_2 \times F)$. Observe also that, due to Lemma~\ref{lemma:geodesics_go_to_quasi_geodesics}, the boundary map $\partial \tilde H$ defines a one-to-one correspondence between the boundaries at infinity. 
Finally, Thurston's argument for the continuity of $\partial \tilde H$ along the horizontal leaves $\tilde M_1 \times \{y\}$ is again based on Morse lemma (see \cite[Corollary 5.9.5]{ThurstonNewBook} and \cite[Proposition 2.10]{Lucker}). This argument follows from Thurston's Lemma~5.9.4, which is also used to prove the quasi-conformality of the boundary map and recalled below as Lemma \ref{lem:keylemma}.
As the leafwise Riemann metrics $\tilde g_1$ on $\tilde M_1 \times F$ and $\tilde g_2$ on $\tilde M_2 \times F$ depend continuously on $y \in F$, the same applies to the boundary map $\partial \tilde H$. Indeed, the identity from $\tilde M_1 \times \{y\}$ onto $\tilde M_1 \times \{y_0\}$ is a quasi-isometry with constants $C_1$ and $C_2$ close to $1$ and $0$. The same happens for the $\tilde M_2 \times \{k(y)\}$ onto $\tilde M_2 \times \{k(y_0)\}$. Combining these remarks with the continuity of the restriction of $\partial \tilde H$ from $\tilde M_1 \times \{y_0\}$ onto 
$\tilde M_2 \times \{k(y_0)\}$, proven by Thurston, we conclude that $\partial \tilde H$ is continuous, as we wanted. 
\end{proof}

The main result of this step is directly based on Corollary 5.9.6 from Thurston's notes \cite{ThurstonNewBook}. This is also the first main point of the original proof by Mostow \cite{Mostow}. Before stating and proving this result, we first need to recall the precise statement of Lemma~5.9.4 from Thurston's notes \cite{ThurstonNewBook} and define the notion of  leafwise quasi-conformal map.

\begin{lemma} \label{lem:keylemma}
 Let $f : \tilde M_1 \to \tilde M_2$ be a quasi-isometry between the universal coverings of two closed hyperbolic manifold $M_1$ and $M_2$. Let $\alpha$ be a geodesic in $\tilde M_1$ and $P$ some hyperplane orthogonal to $\alpha$. Let $\beta$ be the geodesic in $\tilde M_2$ which remains a bounded distance from $f(\alpha)$. Then there exists a constant $C$ depending only on the quasi-isometry constants of $f$ such that the projection of $f(P)$ onto $\beta$ has diameter $\leq C$.
\end{lemma} 

\begin{definition} \label{def:quasi-conformal} A foliated homeomorphism $h$ from a foliated space $\X$ endowed with a leafwise distance $d$ to itself (with a eventually different leafwise distance also denoted $d$) is \emph{leafwise quasi-conformal} if for all point $z \in \X$, the limit 
$$
\lim_{r \to 0} \frac{\sup_{z^\prime,z^{\prime\prime} \in S_r(z)} d\big(h(z^\prime),h(z^{\prime\prime})\big)}{\inf_{z^\prime,z^{\prime\prime} \in S_r(z)} d\big(h(z^\prime),h(z^{\prime\prime})\big)} < +\infty,
$$
where $S_r(z)$ is the sphere of radius $r$ around $z$ contained in the leaf passing through $z$, and $z^\prime$ and $z^{\prime\prime}$ are diametrically opposite. If the deviation from conformality admits the same upper bound for all leaves, we will say $h$ is \emph{uniformly quasi-conformal}.
\end{definition}

\begin{proposition} \label{prop:quasi-conformal}
The boundary map 
$$
\partial \tilde H : \partial (\tilde M_1 \times F) \to 
\partial (\tilde M_2 \times F)
$$
is equivariantly uniformly quasi-conformal. 
\end{proposition} 

\begin{proof} As announced above, Corollary 5.9.6 from Thurston's notes \cite{ThurstonNewBook} show that, for all $y \in F$, the restriction of the leafwise map $\partial \tilde H$ to 
$\partial^y \tilde M_1$ is a quasi-conformal map from 
$\partial^y \tilde M_1$ onto $\partial^{k(y)} \tilde M_2$. However, we need to see that the deviation from conformality admits an upper bound that does not depend on $y \in F$. To do so, we will simply denote $f$ the map induced by $\tilde H$ from $M_1 \times \{y\}$ onto $M_2 \times \{k(y)\}$. As usual, we can assume that $M_1 \times \{y\}$ is the upper half-space $\H^n$ equipped with the Poincaré metric $\tilde g_1^y$. Then $M_2 \times \{k(y)\}$ is also diffeomorphic to $\H^n$, although it is equipped with a possibly different metric $\tilde g_2^{k(y)}$. We can also assume that $z$ and $f(z)$ correspond to the origin $0$. Now we consider any hyperplane $P$ that encircles a domain $D$ around $0$ and is orthogonal to the geodesic $\alpha$ from $0$ to the point at infinity. By Lemma~\ref{lem:keylemma}, the diameter of the projection of $f(P)$ onto $\beta$ (which can be assumed equal to $\alpha$) is bounded by a constant $C$ which does not depend on $y \in F$. There exists two hyperplanes $P_1$ and $P_2$ perpendicular to $\alpha$ such that 
$D_1 \subset f(D) \subset D_2$ and whose distance along $\alpha$ is uniformly bounded for all hyperplanes $P$ and for all transverse coordinates $y \in F$. In fact, this distance is equal to 
$$
log \; r = log \; \frac{R_2}{R_1}
$$
where $R_1$ and $R_2$ are the radii of the boundaries of $P_1$ and $P_2$ into $\partial \H^n$ (see Figure~\ref{fig:conformal}). Therefore the ratio between the maximal and minimal distances is uniformly bounded by a constant $K$ that does not depend on $y \in F$ and the boundary map $\partial f$ is uniformly quasi-conformal. Thus there is a constant $K \geq 1$ (depending on the uniform quasi-isometry constants) such that 
$$
\lim_{r \to 0} \frac{\sup_{(x^\prime,y),(x^{\prime\prime},y) \in S_r(x,y)} d\big(h(x^\prime,y),h(x^{\prime\prime},y)\big)}{\inf_{(x^\prime,y),(x^{\prime\prime},y) \in S_r(x,y)} d\big(h(x^\prime,y),h(x^{\prime\prime},y)\big)} \leq K,
$$
for all $y \in F$ (where we maintain the conventions of  Definition~\ref{def:quasi-conformal}). Finally, from Proposition~\ref{prop:boundarymap}, we know that $\partial \tilde H$ is equivariant with respect to the (leafwise) conformal actions of $\Gamma_1$ and $\Gamma_2$ on $\partial (\tilde M_1 \times F)$ and $\partial (\tilde M_2 \times F)$ respectively. 
\end{proof}

\begin{figure}[b]
  \begin{tikzpicture}[scale=1]
   \clip(-4.4,-3.2) rectangle (4.4,2.2);
    \node at (0,0) {\includegraphics[width=0.8\textwidth]{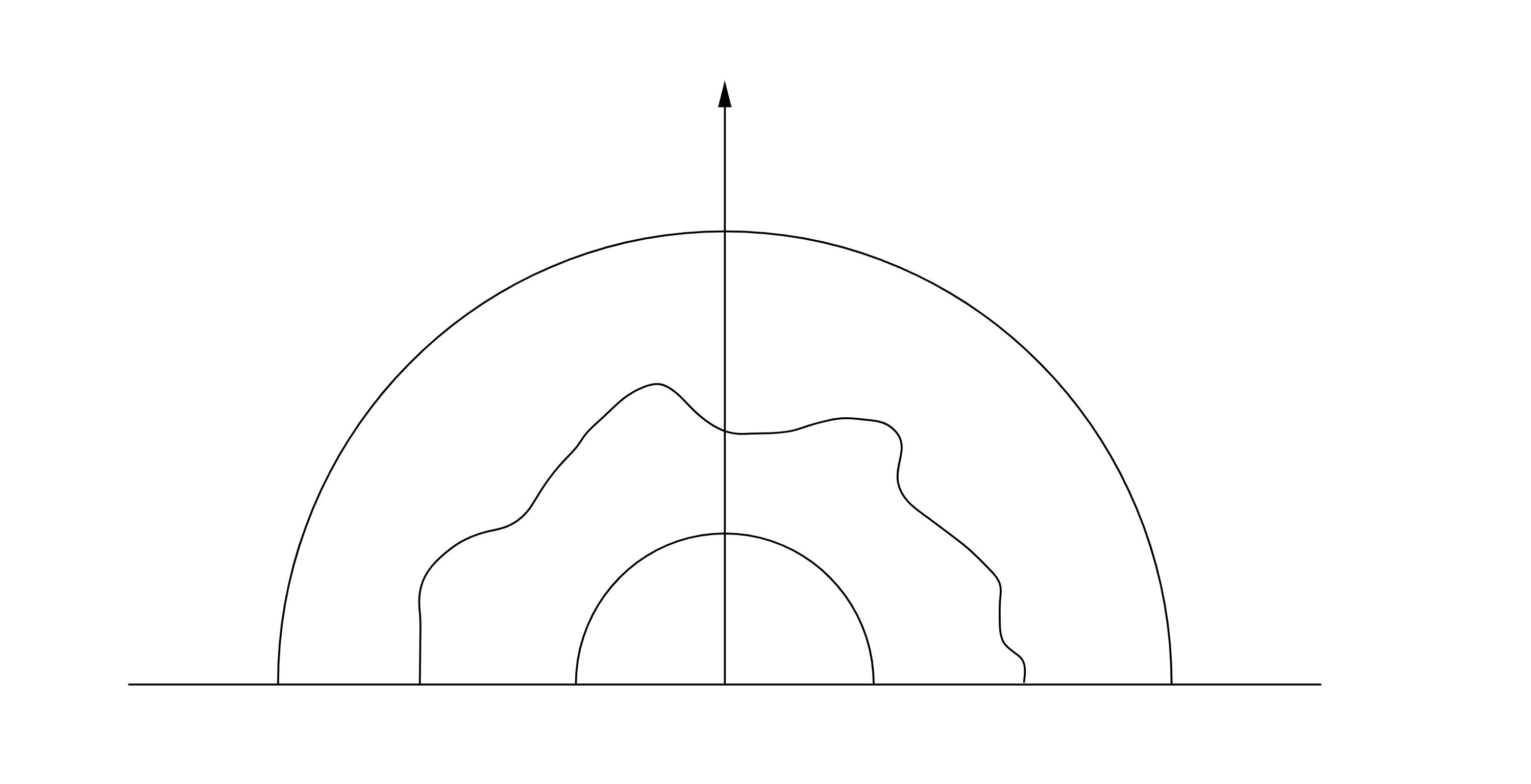}};
    \node at (0.3,1.4) {\scriptsize $\alpha=\beta$}; 
    \node at (1.4,-0.5) {\scriptsize $f(P)$}; 
    \node at (2.2,.2) {\scriptsize $P_2$}; 
    \node at (0.3,-1.5) {\scriptsize $P_1$}; 
    \node at (4.1,-1.9) {\scriptsize $\partial \H^n$}; 
\draw [decorate,decoration={brace,amplitude=3pt,mirror,raise=1ex}]
  (-0.25,-1.9) -- (0.7,-1.9) node[midway,yshift=-3ex]{\tiny  $R_1$};
\draw [decorate,decoration={brace,amplitude=3pt,mirror,raise=1ex}]
  (-0.25,-2.45) -- (2.65,-2.44) node[midway,yshift=-3ex]{\tiny  $R_2$};
  \end{tikzpicture}

  \caption{Bounding the deviation from conformality}
  \label{fig:conformal}
\end{figure}
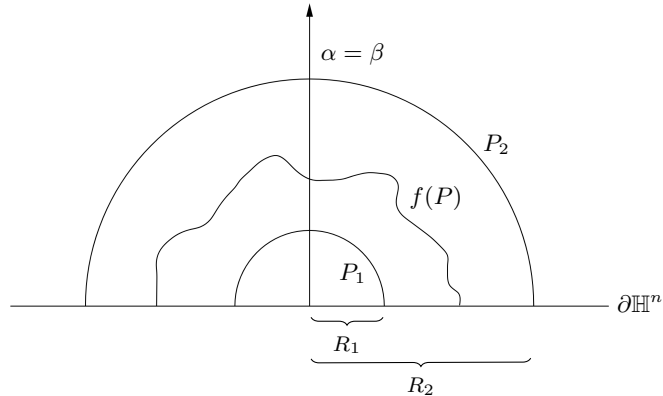

\begin{remark} \label{rem:tubes}
The three first steps of the proof of Theorem~\ref{thm:Mostowsuspension} can be modified by appealing to uniformisation maps $\varphi_1 : \H^n \times F \to M_{\rho_1}$ and 
$\varphi_2 : \H^n \times F \to M_{\rho_2}$ of the foliated spaces $M_{\rho_1}$
and $M_{\rho_1}$. Now both products $\H^n \times F$ are endowed with the same leafwise Riemannian metric obtained from the Poincaré metric on the first factor $\H^n$. The maps $\varphi_1$ and $\varphi_2$ are leafwise smooth and locally isometric continuous maps onto $M_{\rho_1}$ and $M_{\rho_2}$. In other terms,  $\varphi_1$ and $\varphi_2$ are \emph{normal tubes} (in the sense of \cite{ADMV2}) that cover the whole foliated spaces. 

The first step is specific to suspensions, that is, fibre bundles with transverse lamination. As in the second step, the homotopy equivalence $h$ lifts to an equi\-variant quasi-isometry 
$$\tilde H : \H^n \times F \to \tilde H^n \times F$$ 
with uniform constants. However, this map does not necessarily preserve the first projection $p_1 : \H^n \times F \to \H^n$ (and then it does not necessarily induce a fibre bundle morphism between $M_{\rho_1}$ and $M_{\rho_2}$). 
An advantage of this approach lies in the definition of the boundary at infinity, which can be canonically identified to the product $S^{n-1} \times F$. Once again, as a consequence of Morse lemma, the boundary map 
$$
\partial \tilde H : S^{n-1} \times F \to S^{n-1} \times F
$$
is globally continuous. The main advantage of this approach is that it simplifies Definition~\ref{def:quasi-conformal} and the proof of Proposition~\ref{prop:quasi-conformal} in the same way as can be seen in Thurston's proof, namely, we can use the same distance on $\partial (\tilde M_1 \times F)$ and $\partial (\tilde M_2 \times F)$. 

In summary, our first approach uses compatibility with the fibre bundle structure, but this compatibility is not necessary to obtain a quasi-conformal boundary map. So both approaches are interchangeable. 
Here we will continue to use the first one, although the second one, based on normal tubes, will be convenient for extending Theorem~\ref{thm:Mostowsuspension} to skew solenoidal manifolds in Section~\ref{Snonhomogeneous}. 
\end{remark}

\subsection*{The boundary map is conformal} This step essentially follows the last step in Thurston's proof of rigidity theorem. See also the original proof by Mostow \cite{Mostow}. We recall the main points: 
\smallskip 

\noindent
(i) the boundary map is leafwise differentiable a.e.
\smallskip 

\noindent
(ii) the excentricity function (which measures the deviation from a leafwise conformal map) is uniformly bounded a.e. 
\smallskip 

\noindent
(iii) the excentricity function is constant a.e. 
\smallskip 

\noindent
(iv) finally, the boundary map  is conformal. 
\medskip 

\noindent
(i) First, according to Rademacher-Stepanov theorem (see \cite{Vaisala}), quasi-conformal maps have the following strong regularity property: \emph{any $K$-quasi-conformal map  $\partial \tilde H : \partial^y \tilde M_1 \to \partial^{k(y)} \tilde M_2$ is absolutely continuous with respect to the Lebesgue measure and it is differentiable almost everywhere with $K$-quasi-conformal differential.} In our foliated context, we have assumed that the action of $\rho_1(\Gamma_1)$ preserves an ergodic probability measure $\mu_1$ on $F$ with full support (and therefore the action of $\rho_2(\Gamma_2)$ also preserves its image $\mu_2 = k_\ast \mu_1$ on $F$ also with full support). Thus we can formulate the following leafwise version of this regularity property:

\begin{proposition} \label{prop:Rademacher}
The boundary map 
$$
\partial \tilde H : \partial (\tilde M_1 \times F) \to 
\partial (\tilde M_2 \times F)
$$
is leafwise differentiable a.e. with respect to the measure $\nu_1$ on $\partial (\tilde M_1 \times F)$ obtained by integrating the Lebesgue measure $\lambda$ on the universal sphere against an ergodic invariant measure $\mu_1$ on $F$ with full support \hfill \qed
\end{proposition}

Observe that, according to the proof of Proposition~\ref{prop:quasi-conformal}, the boundary map $\partial \tilde H$ is actually leafwise differentiable in restriction to a subset of full measure that intersects $\partial^y\tilde M_1$ in a set of full Lebesgue measure for all $y\in F$.
\medskip

\noindent
(ii) For each point $y \in F$ and for each point $z = (x,y) \in \partial^y \tilde M_1$ such that the map $ \partial \tilde H$ is differentiable at $z$, the differential map $(\partial \tilde H)_{\ast z}$ sends a sphere around the origin to an ellipsoid. The ratio of the lengths of the major and minor axes of this ellipsoid is the \emph{eccentrity} $e(z)$ of the map $\partial \tilde H$ at $z$. As a consequence of Proposition~\ref{prop:quasi-conformal}, we deduce that 
$$
e : \partial (\tilde M_1 \times F) \to \R
$$
is a measurable function which is invariant by the (leafwise) conformal action of $\Gamma_1$ on $\partial (\tilde M_1 \times F)$. If $\partial \tilde H$ is (leafwise) $K$-quasi-conformal, then $e(z) \leq K$ 
for almost every $z\in \partial^y\tilde M_1$ and every $y \in F$.
\medskip 

\noindent
(iii) As next point, we will prove that $e$ is constant a.e. (and later, in the last point, that $K=1$) from the following leafwise version of a classical theorem by E. Hopf \cite{Hopf}: 

\begin{theorem} \label{thm:Hopftheorem}
Let $M_1$ be a connected closed hyperbolic manifold of dimension $n \geq 2$ with fundamental group $\Gamma_1$. 
Let $M_{\rho_1}$ be the suspension of a representation $\rho_1 : \Gamma_1 \to G$ into a homeomorphism group $G$ acting on a compact space $F$. If $\rho_1(\Gamma_1)$ preserves an ergodic probability measure on $F$ with full support, then the foliated geodesic flow $g_t$ on the unit tangent bundle $M_\rho^1$ is ergodic. 
\end{theorem}

\begin{proof}
By Furstenberg's equivalence \cite{Furstenberg2}, as $\isomH$ is unimodular, the holonomy group $\rho_1(\Gamma_1)$ admits an ergodic invariant measure $\mu_1$ on $F$ if and only if the right action of $\isomH$ on $\M_{\rho_1} = \Gamma_1 \bs \big( \isomH \times F \big)$ admits an ergodic invariant measure. By Moore's Ergodicity Theorem \cite[Theorem 2.2.15]{Zimmer}, the right action of the Cartan subgroup $A < \isomH$ on $\M_{\rho_1}$ is still ergodic. This implies that the right action of $A$ on $M^1_{\rho_1} =  \Gamma_1 \bs \big( \isomH/\SO{n-1} \times F \big)$ is ergodic. In other words, the foliated geodesic flow $g_t$ on $T^1 \L_{\rho_1}$ is ergodic.
\end{proof}

Recall that, in Hopf coordinates, the unit tangent bundle $T^1\H^n$ of hyperbolic space becomes 
$$(S^{n-1}_\infty\times S^{n-1}_\infty-\Delta)\times \R$$ 
where $\Delta$ is the diagonal in the product of spheres. The Liouville measure is the product of the measure $dt$ on the factor $\R$ and a measure on $S^{n-1}_\infty\times S^{n-1}_\infty-\Delta$ that is finite on compact sets and in the Lebesgue measure class.

Using again Hopf coordinates in the leafwise direction, we see that the quotient of the product space 
$$\isomH/\SO{n-1} \times F$$ 
by the right action of $A$ identifies to the space 
$$
\partial^2 (\tilde M_1 \times F) = \bigcup_{y \in F} \partial^y \tilde M_1 \times \partial^y \tilde M_1 - \Delta
$$
that fibres over $F$ with fibre $S^{n-1}_\infty \times S^{n-1}_\infty - \Delta$. By Furstenberg's equivalence, the ergodicity of the right action of $A$ on $M^1_{\rho_1}$ is equivalent to the ergodicity of the left action of $\Gamma_1$ on $\partial^2 (\tilde M_1 \times F)$ endowed with the product of a measure that is absolutely continuous with respect to Lebesgue on $S^{n-1}_\infty \times S^{n-1}_\infty - \Delta$ and an ergodic invariant probability measure $\mu_1$ on $F$ with full support. Up to a constant, this measure is induced by the  Liouville measure on the unit tangent bundle $M^1_\rho$.

Now we deduce that every $\Gamma_1$-invariant Borel subset of $\partial (\tilde M_1 \times F)$ has zero measure or full measure.
We do not need the uniqueness of $\mu_1$, although we know that it is true when $G$ acts isometrically on $F$ (including the case where $F=G$). Now, as $e$ is invariant by the action of $\Gamma_1$, we can state the following result (see \cite[Corollary 5.9.9]{ThurstonNewBook}:

\begin{corollary} \label{cor:conformal} The eccentricity map
$$
e : \partial (\tilde M_1 \times F) \to \R
$$ is constant a.e. \hfill \qed 
\end{corollary} 

\noindent
(iv) To conclude, we will have to prove that $e$ is equal to $1$ a.e. In his proof of the classical rigidity theorem, Thurston uses the ergodicity of the action of $\Gamma_1$ on $\partial^2 \tilde M_1$ to assert that there is no invariant measurable line field on $S^{n-1}$ when $n=3$. Considering frame fields instead of line fields, L\"ucker gives in \cite{Lucker} a complete proof of Thurston's claim. Other authors appeal to the results of D. Sullivan \cite{Sullivan1981} and P. Tukia \cite{TukiaQuasiconformal} about $\Gamma_1$-invariant measurable fields of ellipsoids. See for example \cite[Theorem 6.2]{BourdonMostow}. Here we will slightly modify Thurston's proof appealing to an idea from \cite{Bourdon} that again uses Moore's Ergodicity Theorem: 

\begin{proposition} \label{prop:conformal} 
The boundary map 
$$
\partial \tilde H : \partial (\tilde M_1 \times F) \to 
\partial (\tilde M_2 \times F)
$$
is a leafwise conformal diffeomorphism.
\end{proposition}

\begin{proof} By Corollary~\ref{cor:conformal}, we know that $e$ is equal a.e. to a constant $K$. Now, as we announced above, we will prove that $K = 1$.
To do so, we will see that the differential of the boundary map $\partial \tilde H$ is conformal in restriction to a subset of full measure. 
%that project on the whole $F$. 
According to \cite[Theorem 1.5]{Bourdon}, the boundary map $\partial \tilde H$ becomes a
%leafwise 
M\"obius homeomorphism and hence a 
%leafwise 
conformal diffeomorphism 
in restriction to any leaf contained in that set --that is, it is conformal from $\partial^y\tilde M_1$ to $\partial^{k(y)}\tilde M_2$ for a.e. $y\in F$. On the other hand, we know that it is an homeomorphism from $\partial(\tilde M_1\times F)$ to $\partial(\tilde M_2\times F)$. By continuity, it must send spheres in $\partial^y\tilde M_1$ to spheres in $\partial^y\tilde M_2$ for every $y\in F$. Namely, the map $\partial \tilde H$ is globally leafwise conformal.

Indeed, the boundary at infinity $\partial (\tilde M_1 \times F)$ is leafwise diffeomorphic to the homogeneous space $\isomH / P \times F$ where $P = \SO{n-1}AU$ is the minimal parabolic subgroup of $\isomH$ and $U$ is the stable horospherical subgroup of $\isomH$. Its leafwise tangent bundle is diffeomorphic to $\isomH / \SO{n-2}AU \times F$. We will denote $\partial^1 (\tilde M_1 \times F)$ this homogeneous space. By Moore's Ergodicity Theorem \cite[Theorem 2.2.15]{Zimmer}, the right action of the subgroup $\SO{n-2}AU < \isomH$ on $\M_{\rho_1}$ is ergodic. 
Then $\Gamma_1$ acts ergodically on 
$$ 
\partial^1 (\tilde M_1 \times F) = \isomH / \SO{n-2}AU \times F
$$ 
endowed with the product of the Haar measure on the first factor $\isomH / \SO{n-2}AU$ and the invariant probability measure $\mu_1$ on the second one $F$. For each unit tangent vector $v \in \partial^1 (\tilde M_1 \times F)$ based at a generic point $z \in \partial (\tilde M_1 \times F)$, the stretch function 
$$
f(v) = \frac{\norm{(\partial \tilde H)_{*z}(v)} }{\norm{(\partial \tilde H)_{*z}}}
$$
is invariant under the action of $\Gamma_1$ and hence is constant a.e.
This means that the differential of $\partial \tilde H$ is conformal a.e. 
We can fix the points $y \in F$ and restrict the differential to the fibres $\partial^y \tilde M_1$, thus being conformal a.e. As has already been noted, according to 
\cite[Theorem 1.5]{Bourdon}, the map $\partial \tilde H$ is a leafwise M\"obius homeomorphim, and then a leafwise conformal diffeomorphism.
\end{proof}

Finally, to completes the proof of Theorem~\ref{thm:Mostowsuspension}, we need to prove the following proposition: 

\begin{proposition} \label{prop:continuousisometry} 
The boundary homeomorphism 
$$
\partial \tilde H : \partial (\tilde M_1 \times F) \to \partial (\tilde M_2 \times F)
$$ 
extends to an equivariant leafwise isometric homeomorphism
$$
\tilde I : \tilde M_1 \times F \to \tilde M_2 \times F
$$
which is homotopic to 
$$
\tilde H : \tilde M_1 \times F \to \tilde M_2 \times F
$$
by an integral homotopy. Thus $\tilde I$ induces a homeomorphicm
$$I : M_{\rho_1} \to M_{\rho_2}$$ 
which is leafwise isometric and integrable homotopic to $h$.
\end{proposition} 

\begin{proof} 
As the boundary homeomorphism 
$$
\partial \tilde H : \partial (\tilde M_1 \times F) \to \partial (\tilde M_2 \times F)
$$ 
is leafwise conformal by Proposition~\ref{prop:conformal}, it extends to a map 
$$
\tilde I : \tilde M_1 \times F \to \tilde M_2 \times F
$$
which is an isometry in restriction to each horizontal leaf of $\tilde M_1 \times F$. Each of these isometries is also homotopic to the corresponding restriction of the homeomorphism 
$$
\tilde H : \tilde M_1 \times F \to \tilde M_2 \times F
$$
preserving the horizontal lamination (defined by the projection on the second factor $F$). Our aim is to prove that $\tilde I$ is continuous, so \emph{a fortiori} a homeomorphism preserving the projection on the second factor $F$. To do so, we shall consider a \emph{tube of geodesic rays} in $\tilde M_1 \times F$, i.e. a continuous map 
$$
\tau : [0,+\infty) \times F \to \tilde M_1 \times F
$$
such that for each point $y \in F$, the curve $\tau^y : [0,+\infty) \to \tilde M_1 \times \{y\}$ given by $\tau^y(t) = \tau(t,y)$ is a geodesic ray (relatively to the Riemannian metric $\tilde g_1^y$) and its initial tangent vector depends continuously on $y$. Then for each point $y \in F$, the end point $\tau^y(+\infty) \in \partial^y \tilde M_1$ of the ray $\tau^y$ also depends continuously on $y$.
Now our aim reduces to see that, reparametrising $F$ if necessary, the composition $\tilde I \circ \tau$ is again a tube of geodesic rays in $\tilde M_2 \times F$. If we denote by $z$ the image of $\tilde I(0,y)$ by the projection on the second factor $F$, then the restriction of $\tilde I \circ \tau$ to 
$[0,+\infty) \times \{y\}$ is a geodesic ray in $\tilde M_2 \times \{z\}$ that depends continuously on $z$ (since the Riemannian metric $\tilde g_2^z$ depend continuously on $z$). As the homeomorphism induced by $\tilde I$ to the boundary $\partial (\tilde M_1 \times F)$ is equal to $\partial \tilde H$, its end point also depends continuously on $z$, so its initial tangent vector still depends continuously on $z$. Note however that the continuous leafwise isometry 
$\tilde I : \tilde M_1 \times F \to \tilde M_2 \times F$
does not necessarily preserve the projection on the first factor. Thus the homeomorphism 
$$I : M_{\rho_1} \to M_{\rho_2}$$ 
induced by $\tilde I$ is leafwise isometric and integrable homotopic to $h$, but there is no reason to get a bundle isomorphism in general. 
\end{proof} 

\begin{remark} \label{rem:othertubes}
In Proposition~\ref{prop:continuousisometry}, if preferred, tubes of  geodesic rays can be replaced by  tubes of \lq\lq half-planes". Just like in Thurston's notes \cite{ThurstonNewBook} and L\"ucker's thesis \cite{Lucker}, these are regions which are bounded by a hyperplane perpendicular to a geodesic ray. In fact, as proposed in Remark~\ref{rem:tubes}, the use of global uniformisation maps simplifies the proof of Proposition~\ref{prop:continuousisometry}, reducing the foliated context to the classical context. Indeed, in this approach, the boundary map 
$$\partial \tilde H : S^{n-1} \times F \to S^{n-1} \times F$$ 
is given by 
$$\partial \tilde H (x,y) = (\tilde h (x,y), k(y))$$
for all $x \in S^{n-1}$ and all $y \in F$. Here $\tilde h$ is a conformal map in $x$ (depending continuously on $y$) and $k$ is a homeomorphism. Then $\partial \tilde H$ 
extends immediately to a leafwise isometry 
$$\tilde I : \H^n \times F \to \H^n \times F$$ 
given by 
$$\tilde I (x,y) = (\tilde \imath (x,y), k(y))$$
where $\tilde \imath : \H^n \times \{y\} \to \H^n$ is an isometric extension of the conformal map $\tilde h : S^{n-1} \times \{y\} \to S^{n-1}$, both depending continuously on $y$.
\end{remark}

Now the proof of Theorem~\ref{thm:Mostowsuspension} is complete. \qed 

\section{Homogeneous solenoidal manifolds} 
\label{Shomogeneous}

Here we deal with a notion that is apparently more general than the notion of suspension, although both notions are equivalent for solenoidal manifolds and hence Theorem~\ref{thm:Mostowinvlimit} also applies in this context.

\subsection*{Lie-homogeneous hyperbolic solenoidal manifolds} 

Let $\Gamma$ be a cocompact discrete subgroup of the product $\isomH \times G$ of the Lie group $\isomH$ and a Cantor group $G$. The diagonal action of $\Gamma$ on $\H^n \times G$ is free and properly discontinuous and hence the quotient
$$
\X = \Gamma \bs \big( \H^n \times G \big) =  
\Gamma \bs \big(\isomH/\SO{n} \times G \big) 
$$
is a solenoidal manifold of dimension $n$. Since $G$ acts by right translations on $\X$, we deduce $\X$ is topologically homogeneous  (homeomorphic to a McCord solenoid according to \cite{ClarkHurder}). Similarly, the unit tangent bundle
$$
\X^1= \Gamma \bs \big( T^1\H^n \times G \big) =  
\Gamma \bs \big(\isomH/\SO{n-1}  \times G \big) 
$$
is a topologically homogeneous solenoidal manifold of dimension $2n-1$. The homogeneous manifold 
$$
\E =  \Gamma \bs \big( \isomH \times G \big) 
$$ 
admits a double structure of principal $\SO{n}$-bundle over $\X$ and principal $\SO{n-1}$-bundle over $\X^1$. We will say $\X$ is a \emph{Lie-homogeneous hyperbolic solenoidal manifold}. 

\subsection*{Riemann-homogeneous hyperbolic solenoidal manifolds} 
Let $\Gamma$ be a cocompact discrete subgroup of the product $\isomH \times G$ where $G = \Isom (F)$ is the group of isometries of the Cantor set $F$. If $\Gamma$ is torsion-free, then the diagonal action of $\Gamma$ on $\H^n \times F$ is free and properly discontinuous and hence the quotient
$$
\X = \Gamma \bs \big( \H^n \times F \big) =  
\Gamma \bs \big(\isomH/\SO{n} \times F \big) 
$$
is a solenoidal manifold of dimension $n$. The unit tangent bundle
$$
\X^1= \Gamma \bs \big( T^1\H^n \times F \big) =  
\Gamma \bs \big(\isomH/\SO{n-1}  \times F \big) 
$$
is again a solenoidal manifold of dimension $2n-1$. The homogenous manifold 
$$
\E =  \Gamma \bs \big( \isomH \times G \big) 
$$ 
admits a double structure of principal $\SO{n}$-bundle over $E = \E/\SO{n}$ and principal $\SO{n-1}$-bundle over $E^1 = \E/\SO{n-1}$.  Thus, by construction, $\X$ is the fibre bundle with fibre $F$ associated to $E = \E/\SO{n}$ and  $\X^1$ is the fibre bundle with fibre $F$ associated to $E^1 = \E/\SO{n-1}$. We will say $\X$ is a \emph{Riemann-homogeneous hyperbolic solenoidal manifold}. 

\begin{theorem} \label{thm:homogeneous}
Let $\X$ be a Lie-homogeneous (resp. Riemann-homogeneous) hyper\-bolic solenoidal manifold with dense leaves. Then $\X$ is the suspension of a representation into a Cantor group (resp. into the group of isometries of a Cantor set).  
\end{theorem} 

\begin{proof} Assume first $\X$ is a Lie-homogeneous hyperbolic solenoidal manifold. Then the restriction of the projection
$$p_1 : \isomH \times G \to \isomH$$ 
to $\Gamma$ induces a epimorphism onto the subgroup $\Gamma_1 = p_1(\Gamma)$ of $\isomH$. By compactness of $G$, its kernel $\Gamma_0 = \Gamma \cap G$ is finite and therefore $\Gamma_1$ is a discrete subgroup of $\isomH$ isomorphic to $\Gamma_0 \bs \Gamma$. Moreover, reducing $\Gamma_1$ if necessary, we can assume that $\Gamma_1$ is torsion-free. The restriction of the projection 
$$p_2 : \isomH \times G \to G$$ 
to $\Gamma$ 
define a representation of $\Gamma$ into $G$ such that $\Gamma_2 = p_2(\Gamma)$ is a dense subgroup of $G$. Since the normalizer of $\Gamma_0$ in $G$ is a closed subgroup of $G$ containing $\Gamma_2$, we deduce $\Gamma_0$ is a normal subgroup of $G$. Therefore the restriction of $p_2$ to $\Gamma$ induces a representation $\rho_1$ of  
$\Gamma_1 = \Gamma_0 \bs \Gamma$
into the quotient group $G_1 = \Gamma_0 \bs G$ with dense image. Then 
$$
\X = \Gamma \bs \big( \H^n \times G \big)
$$
is leafwise isometric to the quotient of 
$$
\big( \Gamma_0 \bs \H^n \big)  \times G_1 = \Gamma_0 \bs \big( \H^n  \times G \big)
$$
by the diagonal action of $\Gamma_1 = \Gamma_0 \bs \Gamma$ defined from the representation $\rho_1$. Thus $\X$ is a McCord solenoid obtained by suspension.

The same proof is still valid for a Riemannian-homogeneous hyperbolic solenoidal manifold, although principal $G$-bundles used above must be replaced by associated fibre bundles with fibre $F$. 
\end{proof}

This theorem remains valid for any Lie and Riemannian hyperbolic foliation whose transverse structure is represented by a compact Lie group $G$. 
%This is also true if $G$ is a compact p-adic Lie group (see \cite{Ratner1998} and references therein). 
However, there exist many examples of Lie and Riemannian homogeneous foliations with hyperbolic dense leaves which are not suspensions. Nevertheless, in a forthcoming paper \cite{AMVMostowfoliations}, we shall prove that Mostow rigidity remains valid for all Lie and Riemannian foliations with hyperbolic dense leaves of dimension $\geq 3$.
%as well for the intermediate situation of p-adic Lie groups.

\section{Skew solenoidal manifolds}
 \label{Snonhomogeneous}

Here we introduce solenoidal manifolds $\X$ whose transverse structure is modelled by the Cantor group $G$, so topologically homogeneous, although they are not necessarily Lie-homogeneous. As before, we also consider the Riemannian case where $G$ is a compact group acting isometrically on the Cantor set. This condition involves a restriction (which will be removed in \cite{AMVMostowfoliations}), but once again reduces Molino's theory to the distinction between the cases where $G$ is the Cantor group or $G$ is the isometry group of the Cantor set. In fact, we will focus on a special type of Lie and Riemannian solenoidal manifolds that we will call \emph{skew solenoidal manifolds}. In this section all solenoidal manifolds will be assumed minimal, that is, with dense leaves. 
\medskip

\begin{definition} \label{def:Lie}
Let $\X$ be a solenoidal manifold of dimension $n$. We will say $\X$ is a \emph{Lie solenoidal manifold} (resp. \emph{Riemannian solenoidal manifold}) if $\X$ admits a box decomposition $\B$ into flow boxes 
$$
\varphi_i : B_i \stackrel{\cong}{\longrightarrow} \bar \D^n \times T_i
$$
with vertical boundary 
$$
\partial^v \! B_i = \varphi_i^{-1} \big( \partial \bar \D^n \times T_i \big)
$$ 
such that the change of chart 
$$
\varphi_{i} \circ \varphi_j^{-1} : \varphi_j (B_i \cap B_j) \to  \varphi_i (B_i \cap B_j)
$$
is given by 
\begin{equation} \label{eq:changecharts}
(\varphi_{ij}(x),g_{ij}y),
\end{equation}
where $x \in \partial \bar \D^n$, $y \in T_i \cap T_j \subset G$, and $g_{ij}$ is an element of the Cantor group $G$ translating $y$ (resp. $y \in T_i \cap T_j \subset F$, and $g_{ij}$ is an isometry in $G = \Isom(F)$ moving $y$).
%ES UNA DEFINICION RESTRICTIVA DE VARIEDAD SOLENOIDAL RIEMANNIANA.
\medskip 

\noindent
Observe that, in both cases, the foliated structure of $\X$ is given by a \v{C}ech 1-cocycle associated to $\B$ with values in $G$. As mentioned above, in the Riemannian case, the existence of such a cocycle is restrictive, since, in general, $g_{ij}$ is only a local isometry. 
\end{definition}
 
The inverse limit of any tower of coverings is an example of Riemannian solenoidal manifold, while it becomes an example of Lie solenoidal manifold if all coverings are regular. This includes the homogeneous solenoidal manifolds described the previous section. 

\subsection*{A natural family: skew solenoidal manifolds.}
Let $\Gamma$ be a discrete group and $\rho : \Gamma \to G$ a representation such that $\rho(\Gamma)$ is a dense subgroup of $G$. Let 
$$
\varphi: \Gamma \times G \to \isomH
$$ 
be a \emph{cocycle} satisfying 
$$
\varphi(\gamma_1\gamma_2,y) = \varphi(\gamma_1,\rho(\gamma_2)y)
 \varphi(\gamma_2,y)
$$
for all $\gamma_1\gamma_1 \in \Gamma$ and all $y \in G$ (cf. \cite[Section 4.2]{Zimmer}). Such a cocycle
allows us to define a free action 
$$
\Gamma \times \big( \isomH \times G \big) \to \big( \isomH \times G \big)
$$
given by 
\begin{equation} \label{eq:nondiagonal}
\gamma.(u,y) = (\varphi(\gamma,y)u,\rho(\gamma)y)
\end{equation}
for all $\gamma \in \Gamma$, all $u \in \isomH$, and all $y \in G$. By construction, the natural projection $p_2 : \isomH \times G \to G$ is $\Gamma$-equivariant.
 
\begin{definition} Two cocycles $\varphi, \varphi' : \Gamma \times G \to \isomH$ are \emph{cohomologous} if there is a map $\psi : G \to \isomH$ such that
$$
\varphi' (\gamma,g) = \psi(\rho(\gamma)g)\varphi(\gamma,g)\psi(g)^{-1}
$$
for all $\gamma \in \Gamma$, and all $g \in G$. A cocycle $\varphi$ is said to be 
\medskip 

\noindent
(i) \emph{nul cohomologous} if it is cohomologous to the cocycle with constant value $I \in \isomH$,
\medskip 

\noindent
(ii)  \emph{untwisted} if there is a representation $\tau : \Gamma \to \isomH$ such that $\varphi$ is cohomologous to the cocycle given by $(\gamma,g) \mapsto \tau(\gamma)$ for all $\gamma \in \Gamma$, and all $g \in G$.
\end{definition}

Assume that the action~\eqref{eq:nondiagonal} is properly discontinuous and cocompact. Then 
$$
\X = \Gamma \bs \big( \H^n \times G \big) =  
\Gamma \bs \big(\isomH/\SO{n} \times G \big) 
$$
is a $n$-dimensional Lie solenoidal manifold with dense leaves, which will be called \emph{skew}. Observe that $\H^n \times G$ is a skew-product over the (uniquely) ergodic action of $\Gamma$ on $G$ (see \cite{Glasner} for a precise definition) although the Lebesgue measure on $\H^n$ is infinite. If the cocycle $\varphi$ is not untwisted, the dynamics of $\X$ is neither Lie-homogeneous (in the sense of Section~\ref{Shomogeneous}) nor homogeneous (in the sense of Remark~\ref{rem:homogeneous}). A general discussion about homogeneity of Lie laminations will be detailed in \cite{AMVMostowfoliations}.
\medskip 

Replacing the Cantor group with the isometry group $G=\Isom(F)$, the previous construction provides a skew Lie solenoidal manifold 
$$
E = \Gamma \bs \big( \H^n \times G \big) =  
\Gamma \bs \big(\isomH/\SO{n} \times G \big). 
$$
Once again, the quotient 
$$
\X = \Gamma \bs \big( \H^n \times F \big) =  
\Gamma \bs \big(\isomH/\SO{n} \times F \big) 
$$
is a \emph{skew Riemannian solenoidal manifold} which is naturally associated to the skew Lie solenoidal manifold $E$ (although they no longer admit fibre bundle structures as in the case of suspensions). We will interpret $E$ as the solenoidal version of the \emph{transverse frame bundle} in Molino's theory \cite{Molinobook}. As before, if $\varphi$ is not untwisted, the dynamics of $\X$ is neither Riemann-homogeneous nor homogeneous. 

\subsection*{Mostow rigidity for skew solenoidal manifolds.}
As we have just observed, in both cases Lie and Riemannian, a skew solenoidal manifold $\X$ does not in general admit a fibre bundle structure. However, due to its own construction, $\X$ always admits a global uniformisation such as those appearing in Remark~\ref{rem:tubes}. Therefore, even if the proof of Theorem~\ref{thm:Mostowsuspension} does not apply for skew solenoidal manifolds, the alternative proof described in Remarks~\ref{rem:tubes}~and~\ref{rem:othertubes} allows us to formulate the following rigidity theorem (named Theorem~\ref{thm:intro3} in the introduction): 

\begin{theorem}  \label{thm:Mostowskew}
Let $\Gamma_1$ and $\Gamma_2$ be two discrete groups, and $G$ be a Cantor group or the isometry group of a Cantor set.
Let $\X_1$ and $\X_2$ be two minimal skew hyperbolic solenoidal manifolds of dimension $n \geq 3$ obtained from two  representations 
$$
\rho_1 : \Gamma_1 \to G \text{ and } \rho_2 : \Gamma_2 \to G
$$ 
and two  cocycles
$$
\varphi_1 : \Gamma_1 \times G \to \isomH \text{ and } \varphi_2 : \Gamma_2 \times G \to \isomH.
$$
If $h : \X_1 \to \X_2$ is a homotopy equivalence, then there exists a leafwise isometry $I  : \X_1 \to \X_2$ integrable homotopic to $h$. \qed
\end{theorem} 

\begin{corollary} \label{cor:Teichmullerskew}
The Teichmüller space of any skew hyperbolic solenoidal manifold of dimension $n\geq 3$ contains only one point. \qed
\end{corollary}

\subsection*{Final remark.}
Our results raise two questions about Lie foliated spaces with hyperbolic dense leaves. The aim of the announced second part of this paper \cite{AMVMostowfoliations} is precisely to answer the first question:

%\begin{question} \label{ques:Mostowsolenoidal}
%Is Mostow rigidity valid for all Lie and Riemannian hyperbolic solenoidal manifolds of dimension $n \geq 3$?
%\end{question}

\begin{question} \label{ques:Mostowfoliations}
Is Mostow rigidity valid for all Lie foliations with dense hyperbolic leaves of dimension $n \geq 3$?
\end{question}

\noindent
The two main issues related to this question are the construction of the boundary map and the description of non-homogeneous examples from cocycles. 
%Although the first question can be answered by observing that Lie and Riemannian solenoidal manifolds always admit a global uniformisation, both questions will be derived from a common approach, different from the one used here. The idea is to directly construct Thurston's universal spheres for Lie and Riemannian foliations with dense leaves. 
\medskip 

The second question relates to the appropriate version of the Cartan-Hadarmard theorem for Lie solenoidal manifolds. More precisely, the following question is raised:

\begin{question} \label{ques:CartanHadamard}
Are all Lie solenoidal manifolds virtually skew?  
\end{question}

A forthcoming third part of this paper will focus on this question. A parallel discussion also applies to Riemannian foliations or those laminations whose holonomy pseudogroup is generated by the action of a group of isometries according to Definition~\ref{def:Lie}. 
%{\red Los comentarios previos son ciertos para la versión restrictiva de variedad solenoidal riemanniana (que usa isometrías globales y reduce el pseudogrupo de holonomía al pseudogrupo generado por la acción de un grupo de isometrías), pero no lo son si se usa la definición habitual (con isometrías locales)}


\begin{thebibliography}{10}

\bibitem{ADMV2}
Fernando Alcalde~Cuesta, Fran\c{c}oise Dal'Bo, Matilde Mart\'{\i}nez, and
  Alberto Verjovsky, \emph{Unique ergodicity of the horocycle flow on
  {R}iemannnian foliations}, Ergodic Theory and Dynamical Systems \textbf{40}
  (2020), no.~6, 1459--1479.

\bibitem{ALM2011}
Fernando Alcalde~Cuesta, \'Alvaro Lozano~Rojo, and Marta Macho~Stadler,
  \emph{Transversely {C}antor laminations as inverse limits}, Proc. Amer. Math.
  Soc. \textbf{139} (2011), no.~7, 2615--2630. \MR{2784831}

\bibitem{AMVMostowfoliations}
Fernando Alcalde~Cuesta, Matilde Mart\'{\i}nez, and Alberto Verjovsky,
  \emph{Mostow rigidity for hyperbolic {L}ie foliations}, 2026, Preprint.

\bibitem{AlvarezSmith2022}
S\'ebastien Alvarez and Graham Smith, \emph{Earthquakes and graftings of
  hyperbolic surface laminations}, Int. Math. Res. Not. IMRN (2022), no.~4,
  2824--2860. \MR{4381933}

\bibitem{Besson-Courtois-Gallot}
G\'erard Besson, Gilles Courtois, and Sylvestre Gallot, \emph{Minimal entropy
  and {M}ostow's rigidity theorems}, Ergodic Theory Dynam. Systems \textbf{16}
  (1996), no.~4, 623--649. \MR{1406425}

\bibitem{Bourdon}
Marc Bourdon, \emph{Quasi-conformal geometry and {M}ostow rigidity}, Institut
  Fourier, 2004.

\bibitem{BourdonMostow}
\bysame, \emph{Mostow type rigidity theorems}, Handbook of group actions.
  {V}ol. {IV}, Adv. Lect. Math. (ALM), vol.~41, Int. Press, Somerville, MA,
  2018, pp.~139--188. \MR{3888688}

\bibitem{BurgosVerjovsky2020}
Juan~Manuel Burgos and Alberto Verjovsky, \emph{Teichm\"uller theory of the
  universal hyperbolic lamination}, Ann. Acad. Sci. Fenn. Math. \textbf{45}
  (2020), no.~1, 577--599. \MR{4056553}

\bibitem{ClarkHurder}
Alex Clark and Steven Hurder, \emph{Homogeneous matchbox manifolds}, Trans.
  Amer. Math. Soc. \textbf{365} (2013), no.~6, 3151--3191. \MR{3034462}

\bibitem{Deroin}
Bertrand Deroin, \emph{Nonrigidity of hyperbolic surfaces laminations}, Proc.
  Amer. Math. Soc. \textbf{135} (2007), no.~3, 873--881. \MR{2262885}

\bibitem{Fedida}
Edmond Fedida, \emph{Sur les feuilletages de {L}ie}, C. R. Acad. Sci. Paris
  S\'er. A-B \textbf{272} (1971), A999--A1001. \MR{0285025}

\bibitem{Furstenberg2}
Harry Furstenberg, \emph{The unique ergodicity of the horocycle flow}, Recent
  advances in topological dynamics ({P}roc. {C}onf. {T}opological {D}ynamics,
  {Y}ale {U}niv., {N}ew {H}aven, {C}onn., 1972; in honor of {G}ustav {A}rnold
  {H}edlund), Lecture Notes in Math., vol. Vol. 318, Springer, Berlin-New York,
  1973, pp.~95--115. \MR{393339}

\bibitem{Ghys-delaHarpe}
\'Etienne Ghys and Pierre de~la Harpe, \emph{Sur les groupes hyperboliques
  d'apr\`es {M}ikhael {G}romov}, Progr. Math., vol.~83, Birkh\"auser Boston,
  Boston, MA, 1990. \MR{1086655}

\bibitem{Glasner}
Eli Glasner, \emph{Ergodic theory via joinings}, Mathematical Surveys and
  Monographs, vol. 101, American Mathematical Society, Providence, RI, 2003.
  \MR{1958753}

\bibitem{Godbillon}
Claude Godbillon, \emph{Feuilletages}, Progress in Mathematics, vol.~98,
  Birkh\"auser Verlag, Basel, 1991, {\'E}tudes g{\'e}om{\'e}triques. [Geometric
  studies], With a preface by G. Reeb. \MR{1120547}

\bibitem{Gromov}
Michael Gromov, \emph{Hyperbolic manifolds (according to {T}hurston and {J}\o
  rgensen)}, Bourbaki {S}eminar, {V}ol. 1979/80, Lecture Notes in Math., vol.
  842, Springer, Berlin, 1981, pp.~40--53. \MR{636516}

\bibitem{Haefligerclassifiant}
Andr\'e Haefliger, \emph{Groupo\"ides d'holonomie et classifiants}, no. 116,
  1984, Transversal structure of foliations (Toulouse, 1982), pp.~70--97.
  \MR{755163}

\bibitem{HectorBC}
Gilbert Hector, \emph{Groupo\"ides, feuilletages et {$C^*$}-alg\`ebres
  (quelques aspects de la conjecture de {B}aum-{C}onnes)}, Geometric study of
  foliations ({T}okyo, 1993), World Sci. Publ., River Edge, NJ, 1994,
  pp.~3--34. \MR{1363718}

\bibitem{Hector-HirschA}
Gilbert Hector and Ulrich Hirsch, \emph{Introduction to the geometry of
  foliations. {P}art {A}}, Aspects of Mathematics, vol.~1, Friedr. Vieweg \&
  Sohn, Braunschweig, 1981, Foliations on compact surfaces, fundamentals for
  arbitrary codimension, and holonomy. \MR{639738}

\bibitem{Hopf}
Eberhard Hopf, \emph{Fuchsian groups and ergodic theory}, Trans. Amer. Math.
  Soc. \textbf{39} (1936), no.~2, 299--314. \MR{1501848}

\bibitem{Lucker}
Adrien L{\"u}cker, \emph{Approaches to {M}ostow rigidity in hyperbolic space},
  Master Thesis EPFL, 2010.

\bibitem{Matsumotoequi}
Shigenori Matsumoto, \emph{The unique ergodicity of equicontinuous
  laminations}, Hokkaido Math. J. \textbf{39} (2010), no.~3, 389--403.
  \MR{2743829}

\bibitem{McCord}
M.~C. McCord, \emph{Inverse limit sequences with covering maps}, Trans. Amer.
  Math. Soc. \textbf{114} (1965), 197--209. \MR{173237}

\bibitem{Molinobook}
Pierre Molino, \emph{Riemannian foliations}, Progress in Mathematics, vol.~73,
  Birkh\"auser Boston, Inc., Boston, MA, 1988, Translated from the French by
  Grant Cairns, With appendices by Cairns, Y. Carri\`ere, \'E. Ghys, E. Salem
  and V. Sergiescu. \MR{932463}

\bibitem{Mostow}
G.~D. Mostow, \emph{Quasi-conformal mappings in {$n$}-space and the rigidity of
  hyperbolic space forms}, Inst. Hautes \'Etudes Sci. Publ. Math. (1968),
  no.~34, 53--104. \MR{236383}

\bibitem{MuhlyRenaultWilliams}
Paul~S. Muhly, Jean~N. Renault, and Dana~P. Williams, \emph{Equivalence and
  isomorphism for groupoid {$C^\ast$}-algebras}, J. Operator Theory \textbf{17}
  (1987), no.~1, 3--22. \MR{873460}

\bibitem{Prasad}
Gopal Prasad, \emph{Strong rigidity of {${\bf Q}$}-rank {$1$} lattices},
  Invent. Math. \textbf{21} (1973), 255--286. \MR{385005}

\bibitem{Reinhart}
Bruce~L. Reinhart, \emph{Foliated manifolds with bundle-like metrics}, Ann. of
  Math. (2) \textbf{69} (1959), 119--132. \MR{0107279}

\bibitem{Renault1982}
Jean~N. Renault, \emph{{$C\sp{\ast} $}-algebras of groupoids and foliations},
  Operator algebras and applications, {P}art 1 ({K}ingston, {O}nt., 1980),
  Proc. Sympos. Pure Math., vol.~38, Amer. Math. Soc., Providence, RI, 1982,
  pp.~339--350. \MR{679714}

\bibitem{Sacksteder}
Richard Sacksteder, \emph{Foliations and pseudogroups}, Amer. J. Math.
  \textbf{87} (1965), 79--102. \MR{0174061 (30 \#4268)}

\bibitem{Sullivan1981}
Dennis Sullivan, \emph{On the ergodic theory at infinity of an arbitrary
  discrete group of hyperbolic motions}, Riemann surfaces and related topics:
  {P}roceedings of the 1978 {S}tony {B}rook {C}onference ({S}tate {U}niv. {N}ew
  {Y}ork, {S}tony {B}rook, {N}.{Y}., 1978), Ann. of Math. Stud., vol. No. 97,
  Princeton Univ. Press, Princeton, NJ, 1981, pp.~465--496. \MR{624833}

\bibitem{Thurston2}
William~P. Thurston, \emph{Hyperbolic structures on 3-manifolds, ii: Surface
  groups and 3-manifolds which fiber over the circle},  (1986), \url{
  arXiv:math/9801045}.

\bibitem{ThurstonNewBook}
\bysame, \emph{The geometry and topology of three-manifolds. {V}ol. {IV}},
  American Mathematical Society, Providence, RI, 2022. \MR{4554426}

\bibitem{TukiaQuasiconformal}
Pekka Tukia, \emph{On quasiconformal groups}, J. Analyse Math. \textbf{46}
  (1986), 318--346. \MR{861709}

\bibitem{Vaisala}
Jussi V\"ais\"al\"a, \emph{Lectures on {$n$}-dimensional quasiconformal
  mappings}, Lecture Notes in Mathematics, vol. Vol. 229, Springer-Verlag,
  Berlin-New York, 1971. \MR{454009}

\bibitem{VerjovskySullivan}
Alberto Verjovsky, \emph{Low-dimensional solenoidal manifolds}, EMS Surv. Math.
  Sci. \textbf{10} (2023), no.~1, 131--178. \MR{4667418}

\bibitem{Zimmer}
R.~J. Zimmer, \emph{Ergodic theory and semisimple groups}, Monographs in
  Mathematics, vol.~81, Birkh\"auser Verlag, Basel, 1984. \MR{776417}

\end{thebibliography}
\end{document}